\documentclass[a4paper, 11pt]{article}
\usepackage{anysize}
\marginsize{3,5cm}{2,5cm}{2,5cm}{2,5cm}
\usepackage{amssymb}
\usepackage{amsmath,amsbsy,amssymb,amscd}
\usepackage{t1enc}
\usepackage[cp1250]{inputenc}
\usepackage[british]{babel}
\usepackage[all]{xy}
\usepackage{color}
\usepackage{amsfonts}
\usepackage{latexsym}
\usepackage{amsthm}
\usepackage{mathrsfs}
\usepackage{hyperref}
\usepackage{graphicx}
\usepackage{comment}

\usepackage{color}
\usepackage{caption}
\usepackage{subcaption}
\usepackage{enumerate}
\usepackage{xcolor}
\usepackage[draft]{todonotes}

\definecolor{orange}{rgb}{1,0.5,0}

\usepackage{mathrsfs} 

\DeclareMathAlphabet{\mathpzc}{OT1}{pzc}{L}{it} 

\def\diam{\operatorname{diam}}

\theoremstyle{definition}
\newtheorem{definition}{Definition}[section]
\newtheorem{theorem}[definition]{Theorem}

\newtheorem{proposition}[definition]{Proposition}
\newtheorem{corollary}[definition]{Corollary}
\newtheorem{lemma}[definition]{Lemma}

\newtheorem{remark}[definition]{Remark}

\def\Im{\mathrm{Im\,}}

\def\H{\mathbb{H}}
\def\C{\mathbb{C}}
\def\cP{\mathcal{P}}
\def\geq{\geqslant}
\def\leq{\leqslant}
\def\R{\mathbb{R}}
\def\T{\mathbb{T}}

\def\Z{\mathbb{Z}}
\def\N{\mathbb{N}}

\def\id{\mathrm{id}}

\def\Tr{\mathrm{Tr}}
\def\Fix{\mathrm{Fix}}

\def\Q{\mathbb{Q}}

\def\epsilon{\varepsilon}

\def\mf{\mathfrak}
\def\Lie{\operatorname{Lie}}
\def\ad{\operatorname{ad}}
\def\Ad{\operatorname{Ad}}
\def\rank{\operatorname{rank}}
\def\diag{\operatorname{diag}}
\def\End{\operatorname{End}}
\def\inj{\operatorname{inj}}
\def\Kak{\operatorname{Kak}}
\def\Bow{\operatorname{Bow}}
\def\Isom{\operatorname{Isom}}

\def\vol{\operatorname{vol}}

\newcommand{\bea}{\begin{eqnarray}}
  \newcommand{\eea}{\end{eqnarray}}
  \newcommand{\beab}{\begin{eqnarray*}}
  \newcommand{\eeab}{\end{eqnarray*}}

  \newcommand{\be}{\begin{equation}}
  \newcommand{\ee}{\end{equation}}

\newcommand{\set}[1]{\left\lbrace #1 \right\rbrace}
\newcommand{\abs}[1]{\left| #1 \right|}
\newcommand{\mc}{\mathcal}
\newcommand{\norm}[1]{\abs{\abs{#1}}}

\newcommand{\of}{\circ}
\newcommand{\inner}[2]{\left\langle #1, #2 \right\rangle}
\newcommand{\algexp}{\exp_{\operatorname{alg}}}
\newcommand{\geomexp}{\exp_{\operatorname{geom}}}
\newcommand{\alglog}{\log_{\operatorname{alg}}}
\newcommand{\geomlog}{\log_{\operatorname{geom}}}
\newcommand{\mbf}{\mathbf}
\newcommand{\ve}{\epsilon}

\title{Kakutani equivalence of unipotent flows}
\author{Adam Kanigowski \and Kurt Vinhage\footnote{K. V. was supported by the National Science Foundation under Award DMS 1604796} \and Daren Wei\footnote{D. W. was partially supported by the NSF grant DMS-16-02409}}
\date{}

\begin{document}
\maketitle
\begin{abstract}We study Kakutani equivalence in the class of unipotent flows acting on finite volume quotients of semisimple Lie groups. For every such flow we compute the Kakutani invariant of M. Ratner, the value of which being explicitly given by the Jordan block structure of the unipotent element generating the flow. This{,} in particular{,} answers a question of M.\ Ratner. Moreover{,} it follows that the only standard unipotent flows are given by $\begin{pmatrix}1&t\\0&1\end{pmatrix}\times \id$ acting on $(SL(2,\R)\times G')\slash\Gamma'$, where $\Gamma'$ is an irreducible lattice in $SL(2,\R)\times G'$  (with the possibility that $G' = \set{e}$).
\end{abstract}
\section{Introduction}
Classical ergodic theory studies representations of a group $G$ as automorphisms of  measure spaces{:} $g\mapsto T_g$.  For such group actions there is a natural notion of isomorphism: two group actions $(T_g)_{g\in G}$ on $(X,\mu)$ and $(S_g)_{g\in G}$ on $(Y,\nu)$ are {\it (measure theoretically) isomorphic} if there exists a (measure preserving, invertible) map $R:(X,\mu)\to (Y,\nu)$ such that
$$R\of T_g=S_g \of R,\; \mbox{ for all } g\in G.$$
  In what follows,  we will consider the classical setting, where the acting group is $\Z$ or $\R$, corresponding to the iteration of a single automorphism, or flow along a one-parameter measurable family of automorphisms, respectively. Classifying $\Z$ or $\R$ actions up to isomorphism is too difficult of a problem in full generality (see e.g.\ \cite{BFor}, \cite{FWeiss}, \cite{ForRudWeiss}). A much weaker equivalence relation than isomorphism is that of orbit equivalence. We recall, that $(T_g)_{g\in G}$ and $(S_g)_{g\in G}$ are called {\it  orbit equivalent} if there exists a measure preserving, invertible map $R$ taking $(T_g)_{g\in G}$ orbits to $(S_g)_{g\in G}$ orbits (as sets). This notion is however too weak as according to Dye's theorem, \cite{Dye1}, \cite{Dye2} it follows that for $G=\Z$ (or $\R$), any two ergodic measure preserving actions are orbit equivalent.

For $\Z$ and $\R$ actions, an equivalence relation weaker than isomorphism but stronger than orbit equivalence was introduced by S.\ Kakutani \cite{Kak}. Following \cite{Kak}, we say that two $\Z$ actions $T$ and $S$ are {\em Kakutani equivalent} if  there exist measurable sets $A\subset X$ and $B \subset Y$ such that $(T_{|A},A,\mu_A)$ and $(S_{|B},B,\nu_B)$ are isomorphic, where  $T_{|A}$ and $S_{|B}$ denote the corresponding {\em induced isomorphisms} and $\mu_A$ and $\nu_B$ denote the {\em induced measures}. Analogously, we say that two $\R$-actions $(T_t)_{t\in\R}$ and $(S_t)_{t\in\R}$ are {\em Kakutani equivalent} if there exists an
 $L^1(X,\mu)$ {\em time change} of $(T_t)_{t\in \R}$ which is isomorphic with $(S_t)_{t\in \R}$ (see Definition \ref{def:kakeq}). Kakutani originally called this notion weak isomorphism, but as this terminology was later used in a different context, modern treatments use Kakutani equivalence instead.

By Abramov's formula, \cite{Abr}, it follows that Kakutani equivalence preserves the class of zero-entropy, finite entropy systems and infinite entropy systems. In the present paper we study the zero entropy case.

A.\ Katok, \cite{KatMon}, showed that any two ergodic actions with discrete spectrum are Kakutani equivalent. In particular, we call an automorphism $T$ (a flow $(T_t)_{t\in\R}$) {\em standard} 
 or {\it loosely Bernoulli of zero entropy} if it is Kakutani equivalent to an irrational rotation (to a linear flow on $\T^2$).\footnote{Notice that standardness implies ergodicity. By the above result of A.\ Katok, all irrational rotations (linear flows on $\T^2$) are Kakutani equivalent.}  Kakutani originally conjectured that all zero entropy systems were standard (although he did not use this terminology) \cite{Kak}. It turns out that the class of standard systems is quite broad, it contains all systems of local rank one \cite{Fer}  and is closed under factors, inverse limits and compact extensions, \cite{KatMon}, \cite{Rud1}, \cite{Bron}.  Hence, all distal systems are standard, and, in particular, all nil-systems are standard.

The first non-standard system of zero entropy was constructed by J.\ Feldman, \cite{Feld}, by the cutting and stacking method. Later,  A.\ Katok, \cite{KatMon}, and D.\ Ornstein, D.\ Rudolph, and B.\ Weiss, \cite{Rud1}, independently, constructed uncountably many non-Kakutani equivalent zero entropy systems. However, these systems were manufactured to be non-standard and were not systems of general interest. Instead, they were created via certain combinatorial constructions which were later shown to have smooth models. In fact, until now, Kakutani classification of smooth zero-entropy systems which were not created solely for this purpose, has only been answered in a few special cases by M. Ratner. Namely, in \cite{Ratner3}, it is shown that horocycle  flows $(h_t)_{t\in \R}$ acting on finite volume quotients of $SL(2,\R)$ are standard. Then, in \cite{Rat1}, it was shown that $h_t\times h_t$ acting on $SL(2,\R)\slash \Gamma \times SL(2,\R)\slash \Gamma$, the cartesian square of these systems, is not standard, for any (hyperbolic) cocompact lattice $\Gamma$. Finally, in \cite{Rat2}, it was shown that the product of $k$-copies of $(h_t)$ is not Kakutani equivalent to the product of $l$-copies with $k\neq l$. The method in \cite{Rat2} was to introduce, for a general flow $(T_t)$, an invariant of Kakutani equivalence, which was called  the {\em Kakutani invariant} and denoted by $e((T_t),\log)$, which then was estimated to be different for $(h_t)^k$ and $(h_t)^l$.

Notice that these examples come from a very specific class: unipotent flows on quotients of semisimple Lie groups. The study of the Kakutani invariant for these flows was suggested by M.\ Ratner (see Problem 1, \cite{RatICM}). In this class, all previous methods require the use of certain properties of the lattice action. As a result, results were limited to the very restricted class of products of $SL(2,\R)$ with reducible lattices. For many years, the study of the Kakutani equivalence for unipotent flows had no progress in view of these limitations. In fact, since the work of M. Ratner in the 1980s, no progress was made on the question of Kakutani equivalence for any naturally defined systems. The results of the present paper represent the first major step forward in over thirty years for our understanding of Kakutani equivalence of algebraic actions.
We show that for every unipotent flow on a semismiple Lie group quotient, the Kakutani invariant can be explicitly computed from the Jordan block structure of the unipotent element that generates the flow (see Definition \ref{def:chainbasis}). 


There is a remarkable difference between the semisimple and nilpotent cases: for the semisimple case, there is a nontrivial but explicit formula for the Kakutani invariant in terms of the slow entropy. In particular, by Corollary \ref{cor1} there are very few unipotent flows which are standard. In the nilpotent case, the slow entropy can be arbitrarily large, but the resulting systems are always standard. At first glance this may be quite a surprise, since the local behavior of unipotent flows on quotients of semisimple groups  and nilmanifolds are very similar. In fact, there is a unified argument that shows that the slow entropy of these systems does not see the global structure of these groups (see \cite{KanigowskiVinhageWei}). However, in the nilmanifold case, the directions in which the maximal divergence is seen are not mixed with the directions which cause divergence, even after recurrence. These directions are central in the group and descend tori on the nilmanifold. But in the semisimple case, the directions in which divergence are seen are mixed with the remaining directions which cause divergence after they recur.

The proof we implement here is not an adaptation of Ratner's argument in \cite{Rat1} and \cite{Rat2}, which uses specific properties of hyperbolic lattices in $PSL(2,\R)$ and their boundary actions. Instead, we replace it by using {\it multi-scale analysis}, which controls orbits on intermediate scales, combined with a polynomial divergence property, which generalizes the divergence properties of horocycle flows and was first observed in \cite{witte}. This has many advantages over previously used technology, as it works for arbitrary semisimple groups. In particular, we use only very coarse properties of these groups such as exponential volume growth, estimates on the number of lattice points in balls, and the existence of certain renormalizing flows which interact in special ways with the unipotent flows (see Section \ref{sec:sl2}).

\subsection{Statement of Main Results}

In what follows $G$ is a semisimple linear Lie group\footnote{We do not lose much generality in assuming that $G$ is a linear Lie group, as any Lie group is a discrete cover of some matrix group} and $\Gamma$ is a lattice in $G$ (we do not assume that $\Gamma$ is cocompact). Let moreover $\mf g = \Lie(G)$ denote the Lie algebra of $G$. A flow $(\phi_t)$ on $G\slash \Gamma$ is called {\it unipotent}, if $\phi_t$ is the left translation action by $\exp(tU)$, where $U\in \mf g$ is such that ${\ad_U}^k = 0$ for some $k$, where $\ad_U\in \End(\mf g)$ is the adjoint operator, $\ad_U(X)=[U,X]$. The flow $(\phi_t)$ preserves Haar measure $\mu$ on $G\slash \Gamma$. We may also associate a list of numbers $(m_1,\dots,m_n)$ called the {\it chain structure} of $U$ which are the sizes of the Jordan blocks for $\ad_U$ (see Definition \ref{def:chainbasis}). Then, we have the following invariant which is the {\it growth rate} or {\it slow entropy} of $(\phi_t)$:
\be\label{eq:gru}
GR(U):=\frac{1}{2}\sum_{i=1}^n m_i(m_i+1).
\ee
As shown in \cite{KanigowskiVinhageWei}, the number $GR(U)$ describes the asymptotic orbit growth (both in the topological and metric category). Moreover (see Section \ref{uni:flows}), it follows that $GR(U)\geq 3$. The main theorem is the following (see Definition \ref{def:Kakutani}):

\begin{theorem}\label{th1}Let $G$ be a semisimple linear Lie group and
$(\phi_t)= L_{\exp(tU)}$ a unipotent flow on $G\slash \Gamma$. If $\Gamma$ is cocompact, we have
$$
e((\phi_t),\log)=GR(U)-3.
$$
For finite volume $\Gamma$, we have
$$GR(U)-4\leq e((\phi_t),\log)\leq GR(U)-3.$$
Moreover, if $GR(U)=3$, then $(\phi_t)$ is standard.
\end{theorem}

By a direct computation, one gets $3k-4\leq e((h_t)^k,\log)\leq 3k-3$. This, in particular, generalises M.\ Ratner's result, \cite{Rat2} to any  lattice in $SL(2,\R)^k$. If the lattice is additionally cocompact, then $e((h_t)^k,\log)=3k-3$. Theorem \ref{th1} allows one to deduce the following immediately from Lemmas \ref{th2} and \ref{lem:GRge5}:

\begin{corollary}\label{cor1} The only ergodic unipotent flows on finite volume quotients of linear semisimple Lie groups which are standard are of the form $\phi_t=\begin{pmatrix}1&t\\0&1\end{pmatrix}\times \id$ acting on $(SL(2,\R)\times G')\slash \Gamma$,
where $\Gamma$ is irreducible.
\end{corollary}
Theorem \ref{th1}, gives a solution to M.\ Ratner's Problem 1 in \cite{RatICM} (see also \cite{KleinShahStarkov}):
\begin{corollary}\label{cor2}
Let $G$ be a linear semisimple Lie group with $\dim G>3$, and $G / \Gamma$ be a finite volume homogeneous space of $G$.

\begin{enumerate}[(i)]
\item There are ergodic unipotent flows on $G / \Gamma$ which are not standard.
\item If $G$ is simple, no unipotent flow on $G/\Gamma$ is standard.
\item If $G$ has real rank at least two, there are two unipotent flows on $G / \Gamma$ (which are not identity, but not necessarily ergodic) which are not Kakutani equivalent.
\item If $G \cong SL(d,\R)$, then there are at least $d-1$ flows on $G / \Gamma$ which are pairwise non-Kakutani equivalent.
\end{enumerate}
\end{corollary}

In fact, we expect the number of pairwise non-Kakutani equivalent flows on $SL(d,\R)/\Gamma$ to grow on the order of $d^3$ (see Remark \ref{rem:cubic-diff}). A proof of Corollary \ref{cor2} is given in Section \ref{uni:flows}. Moreover, our main result also allows one to construct algebraic examples which answer negatively the following question by A.\ Katok, \cite{KatMon}: if $T \of S=S \of T$ (i.e. $S\in C(T)$) and $T$ is standard, does it follow that $S$ is standard? The first such counterexamples were constructed by de la Rue in \cite{deLaRue}. However, these examples were Gaussian systems which are not known to have smooth finite dimensional models.

\begin{corollary}\label{cor3} Let $T=h_1\times \id$ and $S= h_1 \times h_1$ acting on $SL(2,\R)^2 / \Gamma$ with $\Gamma$ irreducible. Then $S$ and $T$ are ergodic and commute, with $T$ standard and $S$ non-standard.
\end{corollary}

 We finish the introduction with the following questions:

\textbf{Question 1.} When is the flow $(\phi_t)= L_{\exp(tU)}$ acting on $G\slash \Gamma$ Kakutani equivalent with its action on $G\slash \Gamma'$?

Notice that if $\Gamma$ and $\Gamma'$ are conjugated, then the actions are isomorphic and hence Kakutani equivalent. Therefore the interesting case is to consider question 1 for $\Gamma$ and $\Gamma'$ which are not algebraically related.

 The above question is a particular case of the following general question:

\textbf{Question 2.} Let $U\in \mf g$ and $U'\in \mf g'$. Is it true that if $GR(U)=GR(U')$ then the flows $(\phi_t)= L_{\exp(tU)}$ and $(\phi'_t)= L_{\exp(tU')}$ are Kakutani equivalent?

A positive answer to Question 2 would mean that the Kakutani invariant is a full invariant in the class of unipotent flows (the same way as Kolmogorov-Sinai entropy is a full invariant for Bernoulli shifts, \cite{Orn}). Notice also that Question 1 is a special case of Question 2.

Notice that in Theorem \ref{th1}, we use the $\log$ function to compute the Kakutani invariant $e((\phi_t),\log)$. In general (see \cite{Ratner4}), one may consider any  $u:[0,+\infty)\to [0+\infty)$  such that $\lim_{t\to+\infty} \frac{u(at)}{u(t)}=1$, for any $a>0$. We have the following general problem:

\textbf{Problem 1}: For any function $u$ as above  construct a flow $(T_t)$ such that $0<e((T_t),u)<+\infty$.

Notice that it is much harder to construct systems with a prescribed Kakutani invariant than with the Hamming one (this invariant is called slow entropy in \cite{KatThou}). Indeed, it follows from \cite{KatThou} that (for natural systems such as unipotent flows) slow entropy behaves well under taking products, which is not the case for the Kakutani invariant, as is demonstrated by considering $(h_t)$ on $SL(2,\R) / \Gamma$ and $(h_t \times h_t)$, first considered in  \cite{Rat1}.

\textbf{Acknowledgements} The authors would like to thank Anatole Katok for suggesting this problem and his encouragements in its development. The authors are also grateful to  Federico Rodriguez-Hertz, and Jean-Paul Thouvenot for offering their insight on the subject, as well as  Dmitry Dolgopyat and Mariusz Lema\'{n}czyk on their useful comments on a preliminary version of the paper.

\subsection{A Reader's Guide}

We write the paper with readers from two distinct fields in mind: ergodic theory and measurable invariants, and Lie groups and homogeneous dynamics. We therefore include a section to describe some standard tools from each (Sections \ref{sec:def} and \ref{sec:prelims}). In Section \ref{sec:more-defs}, we combine ideas from each of these fields to make definitions which allow us to analyze the decay rate of Kakutani balls. Section \ref{sec:orb} contains  some algebraic lemmas which are applied in Sections \ref{sec:thmref} and \ref{sec:technical}.
Since some proofs have a clear main idea but are technical, we have included outlines of each important reduction (before its proof) to explain what the technicalities mean intuitively.


The key  technique of the paper is developing new counting results for the Kakutani invariant. The main idea is that if two points are Kakutani close (which, in general, is very hard to control), then they are algebraically close on a long block.
 In particular, the first reduction of Theorem \ref{th1} is in Section \ref{sec:thmref} (Theorem \ref{prop:longblock}), which relates Kakutani balls to ``Bowen-like'' balls (Definition \ref{def:kakball}). These are algebraically, and not dynamically, defined and we can obtain good estimates on their decay rates. Therefore, the main purpose of Theorem \ref{prop:longblock} is to relate the dynamically defined Kaktuni balls with a more algebraic definition.

The main difficulty with the Kakutani invariant is that the dynamical criterion for being in a Kakutani ball does not give us control over the full orbit. Therefore, the strength of Theorem \ref{prop:longblock} is the guarantee that this does happen: not only for a large proportion of time do we have closeness of orbits, but also for a very long interval.

The remainder of the paper is dedicated to the proof of Theorem \ref{prop:longblock}. A series of further reductions to prove Theorem \ref{prop:longblock} are made in Section \ref{sec:technical}. The main idea is the following: to guarantee a long interval in which orbits are close, and not just a large proportion of time as guaranteed by the Kakutani condition, one must show that orbits cannot align, separate, and realign in a negligible amount of time on large scales. Proposition \ref{prop:splitlong} is a way to guarantee that this cannot happen: for any matching of orbits, the smaller segments of matching times cannot take up a large portion of the matching interval.

Let us  point out that our technique is different from Ratner's methods from \cite{Rat1}, \cite{Rat2}. Indeed, the methods in \cite{Rat1} and \cite{Rat2} are crucially based on the fact that the lattice is a product of hyperbolic lattices in $SL(2,\R)$. Our method is based on controlling the algebraic (polynomial) divergence of the unipotent flow and not on controlling the behaviour of  the return times using finer properties of the lattice. The details will be explained more in future sections.

\begin{remark}
The only place were we use the fact that $G$ is linear is the computation in Appendix, where we compute products of elements from the $\mf{sl}(2,\R)$-triple in $G$. If $G$ is linear, it follows that the homomorphism $\phi: \mf{sl}(2,\R)\to \mf {g}$ lifts to a homomorphism $\Phi:SL(2,\R)\to G$ (and not just its universal cover, which is all that is guaranteed from general Lie theory), see Lemma \ref{lem:sl2-lift}. This allows us to make computations in $SL(2,\R)$ and conclude things about the corresponding products in $G$.
\end{remark}

\section{Preliminaries on Kakutani Equivalence}\label{sec:def}
In this section we will introduce some basic definitions.
We first recall the definition of Kakutani equivalence. For a flow $(T_t)$ on  $(X,\mathscr{B},\nu)$ and a function $\alpha\in L^1_+(X,\mathscr{B},\nu)$, the flow $(T^\alpha_t)$ is called a {\it time change} of $(T_t)$ (along $\alpha$) if
$$
T^\alpha_t(x)=T_{u(t,x)}(x),
$$
where $u(t,x)$ is a (unique) solution to
$$
\int_0^{u(t,x)}\alpha(T_sx)ds=t.
$$
it follows that $(T^\alpha_t)$ preserves measure $d\bar{\nu}:=\alpha(\cdot)d\nu$.
\begin{definition}[Kakutani equivalence, \cite{Katok1}]\label{def:kakeq}
Two ergodic measure preserving flows $(T_t)$ on $(X,\mathscr{B},\nu)$ and $(S_t)$ on $(\tilde{X},\mathscr{C},\tilde{\nu})$ are {\it Kakutani equivalent}, if $(S_t)$ is isomorphic to $(T^\alpha_t)$ for some $\alpha\in L^1_+(X,\mathscr{B},\nu)$.
\end{definition}

Following \cite{Ratner4}, we will introduce the Kakutani invariant for an ergodic flow $(T_t)$ acting on a Lebesgue space $(X, \mathscr{B},\nu)$. For a finite measurable partition $\cP$ of $X$ and an element $x\in X$, we denote by $\cP(x)$ the atom of $\cP$  containing $x$ and let $I_R(x):=\{T_sx\;:\; s\in[0, R]\}$. Let $l$ denote the Lebesgue measure on $[0,R]$.

\begin{definition}[$(\epsilon,P)$-matchable, \cite{Ratner4}]\label{def:match}
For $x,y\in X$, $\epsilon>0$ and $R>1$, $I_R(x)$ and $I_R(y)$ are called $(\epsilon,\cP)$-matchable if there exists a subset $A=A(x,y)\subset[0,R]$, $l(A)>(1-\epsilon)R$ and an increasing absolutely continuous map $h=h(x,y)$ from $A$ onto $A'=A'(x,y)\subset[0,R]$, $l(A')>(1-\epsilon)R$ such that  $\cP(T_tx)=\cP(T_{h(t)}y)$ for all $t\in A$ and the derivative $h'=h'(x,y)$ satisfies
\begin{equation}
|h'(t)-1|<\epsilon\text{ for all }t\in A.
\end{equation}
 We call $h$ an {\it $(\epsilon, \cP)$-matching} from $I_R(x)$ onto $I_R(y)$.
\end{definition}

The Kakutani invariant is defined based on the above definition.

\begin{definition}[Kakutani invariant, \cite{Ratner4}]\label{def:Kakutani}
Define $$f_R(x,y,\cP)=\inf\{\epsilon>0:\text{$I_R(x)$ and $I_R(y)$ are $(\epsilon, P)$-matchable}\}.$$ Then denote $B_R(x,\epsilon,\cP)=\{y\in X:f_R(x,y,\cP)<\epsilon\}$ as $(R,\cP)$-ball of radius $\epsilon>0$ centered at $x\in X$, $R>1$. A family $\alpha_R(\epsilon,\cP)$ of $(R,\cP)$-balls of radius $\epsilon>0$ is called $(\epsilon,R,\cP)$-cover of $X$ if $\nu(\cup\alpha_R(\epsilon,P))>1-\epsilon$. Denote $K_R(\epsilon,\cP)=\inf|\alpha_R(\epsilon,P)|$ where $|A|$ denotes the cardinality of $A$ and infimum is taken over all $(\epsilon, R, \cP)$-covers of $X$. Let $\mc F$ denote the family of all nondecreasing functions from $\mathbb{R}^+$ onto itself, converging to $+\infty$. For $u\in \mc F$, we denote,
\begin{equation}
\begin{array}{rcl}
\beta(u,\epsilon,P) & = & \displaystyle\liminf_{R\to\infty}\frac{\log K_R(\epsilon, P)}{u(t)};\\
e(u,P) & = & \displaystyle\limsup_{\epsilon\to0}\beta(u,\epsilon,P);\\
e((T_t),u) & = &\displaystyle \sup_{P}e(u,P).
\end{array}
\end{equation}
\end{definition}

We also recall the following theorems, the first one is the generator theorem.

\begin{theorem} [Generator theorem, \cite{Ratner4}]\label{rat:gen}
Let $(T_t)$ be an ergodic measure-preserving flow on $(X,\mathscr{B},\nu)$ and let $\cP_1\leq\cP_2\leq\ldots$ be an increasing sequence of finite measurable partitions of $X$ such that $\vee_{n=1}^{\infty}\cP_n$ generates the $\sigma-$algebra $\mathscr{B}$. Then $e((T_t),u)=\sup_m e(u, \cP_m)$ for all $u\in \mc F$.
\end{theorem}

The following theorem shows that the above quantity is an invariant of Kakutani equivalence.


\begin{theorem}[\cite{Ratner4}]\label{cor:rat}
Let $(T_t)$ and $(S_t)$ be two ergodic Kakutani equivalent measure preserving flows on $(X,\mathscr{B},\nu)$ and $(\tilde{X},\tilde{\mathscr{B}},\tilde{\nu})$. Then
$$e((T_t),u)=e((S_t),u)$$
for all $u\in \mc F$ with
$$\lim_{t\to\infty}\frac{u(at)}{u(t)}=1\text{ for all }a>0.$$
\end{theorem}

Moreover, we have the following theorem (see e.g. \cite{Ratner4}):

\begin{theorem}\label{th:LB}
A zero-entropy ergodic measure preserving flow $(T_t)$ is standard if and only if $e((T_t),u)=0$ for all $u\in \mc F$.
\end{theorem}

We will also use the following definition of matching balls:

\begin{definition}\label{def:balls}Fix $\epsilon>0$, let $x,y\in M$ be $(\epsilon,\cP)$-matchable (see Definition \ref{def:match}) and let $h:A(x,y)\to A'(x,y)$ be an $(\epsilon,\cP)$-matching. For $u\in A(x,y)$ and $L>0$ let
$$
B(u,L):=\{r\in A(x,y)\;:\; r\geq u,\; r-u\leq L\}.
$$
denote the {\it matching ball} around $(u,h(u))$.
\end{definition}
Finally, we give a simple general remark, which we will use in the proof of Theorem \ref{th1}.
\begin{remark}\label{compef} If there exists a set $D\subset X$, such that for every $y\in D$, we have
$$
\mu(B_R(y,\epsilon,\cP)\cap D)\leq a(R,\epsilon),
$$
for $\epsilon<\mu(D)$, then  $K_R(\epsilon/5,\cP)\geq \frac{1}{a(R,\epsilon)}$.

On the other hand if for every $y\in D_\epsilon$, $\mu(D_\epsilon)\geq 1-\epsilon$, we have
$$
\mu(B_R(y,\epsilon,\cP))\geq b(R,\epsilon),
$$
then $K_R(5\epsilon,\cP)\leq \frac{1}{b(R,\epsilon)}$.

We recall also that  $f_t(\cdot,\cdot,\cP)$ does not define a metric (triangle inequality fails), however it is close to a metric: if $x\in
 B_R(y,\epsilon,\cP)$ and $y\in B_R(z,\epsilon,\cP)$, then $x\in
 B_R(z,5\epsilon,\cP)$.

\end{remark}

\section{Preliminaries on Homogeneous Spaces}
\label{sec:prelims}

In this section, we recall some basic facts from the theory of Lie groups and homogeneous spaces. Throughout the paper $G$ will denote a semisimple Lie group with Lie algebra $\mf g$. Given $g \in G$, let $L_g, R_g : G \to G$ denote the left and right translations on $G$. Let $\exp : \mf g \to G$ denote the exponential mapping of the Lie algebra $\mf g$ onto $G$. Then $\exp$ has a local inverse $\log$ sending a neighborhood of $e \in G$ to a neighborhood of $0 \in \mf g$.

\subsection{Metrics and Measures on Homogeneous Spaces}
\label{sec:metrics}

Let $\Gamma \subset G$ be a (discrete) subgroup. We introduce a metric on a the homogeneous space $G/\Gamma$ by first introducing a right invariant metric on $G$. Fix an inner product $\inner{\cdot}{\cdot}_0$ on $\mf g$, and define for $v,w \in T_gG$:

\[ \inner{v}{w} = \inner{dR_{g^{-1}}v}{dR_{g^{-1}}w}_0\]

By construction, $\inner{\cdot}{\cdot}$ is right invariant, so it induces a Riemannian metric on the space $G/\Gamma$. The Riemannian metric also has an associated exponential mapping $\geomexp : \mf g \to G$, which is $C^\infty$ and satisfies

\begin{equation}
d_0\geomexp = \id.
\end{equation}

Like the algebraic exponential, there is a local inverse of $\geomexp$ which we will denote by $\geomlog$. The following is immediate from the definition of the inner product.

\begin{lemma}
The Riemannian volume is a (right) Haar measure on $G$. In particular, it is independent of the metric $\inner{\cdot}{\cdot}_0$ when determining a probability measure on a homogeneous space.
\end{lemma}

\subsection{The Adjoint Representation}

 $G$ acts on itself by conjugation $C_g : h \mapsto g^{-1}hg$, and taking the derivative at the identity in the coordinate $h$ gives the {\it adjoint representation} of $G$ on $\mf g = T_eG$, $\Ad : G \to GL(\mf g)$. Taking the derivative of this map in the $g$ coordinate yields the Adjoint representation of the Lie algebra $\mf g$, $\ad : \mf g \to \End(\mf g)$, which coincides with the Lie bracket: $\ad(X)Y = [X,Y]$. The following are standard tools from the theory of Lie groups, which we write as a Lemma to reference.

\begin{lemma}
	\label{lem:adjoint}
	If $X,Y \in \mf g$,
	
	\[ \algexp(-X)\algexp(Y)\algexp(X)= \algexp(\Ad(\algexp(X))Y) \]
	
	\[ \exp(\ad(X)) := \sum_{k=0}^\infty \dfrac{\ad(X)^k}{k!}  = \Ad(\algexp(X))\]
\end{lemma}

\subsection{Decompositions and Subgroups of Semisimple Groups}

\subsubsection{$\mf{sl}(2,\R)$ triples}
\label{sec:sl2}

Let $V = \begin{pmatrix} 0 & 0 \\ 1& 0 \end{pmatrix}$, $X = \begin{pmatrix} 1 & 0 \\ 0 & -1 \end{pmatrix}$ and $U = \begin{pmatrix} 0 & 1 \\ 0 & 0 \end{pmatrix}$ be the standard generators of $\mf{sl}(2,\R)$. Let $G$ be a simply-connected semisimple Lie group of rank $r$, and $\mf g = \Lie(G)$. We abusively let $U \in \mf g$ denote an arbitrary unipotent element (i.e., an element such that 0 is the only eigenvalue of $U$). This is because given any unipotent element, there exists a homomorphism $\varphi : \mf{sl}(2,\R) \to \mf g$ such that $\varphi(U)$  is this given element. While this homomorphism is not unique, it is unique up to automorphism of $\mf g$ fixing $U$. We therefore identify $\mf{sl}(2,\R)$ with its image under $\varphi$.

Given a subalgebra isomorphic to $\mf{sl}(2,\R)$ of $\mf g$, we may consider the action $\ad : \mf{sl}(2,\R) \to \End(\mf g)$ which maps $X \mapsto \ad_X$. Since it is a subalgebra, this is a representation of $\mf{sl}(2,\R)$. Since $\mf{sl}(2,\R)$ is a semisimple algebra, this representation splits as a sum of irreducible representations. The irreducible representations of $\mf{sl}(2,\R)$ classified up to isomorphism, with classes indexed by $\N$. Let $E_n$ be an $(n+1)$-dimensional real vector space generated by vectors $X_{2k-n}$, $k = 0,\dots, n$. Then there exist nonzero constants $a_{n,k}$ such that:

\begin{eqnarray}
\label{eqn:Uaction} \pi_n(U) X_{2k-n} = X_{2k-n+2} \\
\label{eqn:Xaction}  \pi_n(X) X_{2k-n} = (2k-n) X_{2k-n}\\
\label{eqn:Vaction}  \pi_n(V)X_{2k-n} = a_{n,k} X_{2k-n-2}
\end{eqnarray}

where we assume $U$ sends $X_{n}$ to 0 and $V$ sends $X_{-n}$ to 0. Note that the first three are special cases: $\pi_0$ is the trivial representation, $\pi_1$ is the standard representation and $\pi_2$ is the adjoint representation. Given elements $Y_1,\dots,Y_n \in \mf g$, let $C(Y_1,\dots,Y_n)$ denote the common centralizer of the $Y_i$. That is:

\[ C(Y_1,\dots,Y_n) = \set{ H \in \mf g : \ad_{Y_i}(H) = 0 \mbox{ for all } i=1,\dots,n} \]

The following is a straightforward finite-dimensional version of the Howe-Moore theorem:

\begin{lemma}\label{lem:com} If $A\in C(U,X)$, then $A\in C(U,V,X)$.
\end{lemma}
\begin{proof}
Pick a basis $\set{X_{2i-m_j}^j : i = 0,\dots m_j \mbox{ and } j = 1,\dots, n}$ of $\mf g$ such that $X_{2i-m_j}^j$ span a representation $\pi_{m_j}$ of $\mf{sl}(2,\R)$ with relations determined by \eqref{eqn:Uaction}-\eqref{eqn:Vaction}. We may write $A = \sum a_{i,j}X_{2i-m_j}^j$, and notice that $\ad_X(A) = \sum a_{i,j}(2i-m_j)X_{2i-m_j,j}$. If $\ad_X(A) = 0$, then $a_{i,j} = 0$ unless $2i = m_j$. Furthermore, $\ad_U(A) = \sum a_{i,j} X_{2i-m_j+2,j}$. This implies that $a_{i,j} = 0$ unless $2i - m_j = m_j$. If $a_{i,j} \not= 0$, the first condition implies $i = m_j/2$ and the second implies $i = m_j$. The only way this occurs is when $i = m_j = 0$. That is, $A$ must be a sum of vectors spanning trivial representations, and $V$ must act trivially as well.
\end{proof}

The following lemma allows us to make computations in $SL(2,\R)$ directly for the corresponding elements of $G$:

\begin{lemma}
\label{lem:sl2-lift}
Let $H \subset GL(N,\R)$ be a Lie group and $\phi : \mf{sl}(2,\R) \to \Lie(H)$ be a homomorphism. Then there exists a unique $\widetilde{\phi} : SL(2,\R) \to H$ such that $d\widetilde{\phi} = \phi$
\end{lemma}

\begin{proof}
Notice that $\phi$ induces a representation of $\mf{sl}(2,\R)$ on $\R^N$ which we denote by the same name, since $\phi(X)$ is a matrix in $\Lie(H) \subset \mf{gl}(N,\R)$ by definition. Then we may decompose $\R^N$ as a direct sum of irreducible subrepresentations $\R^N = \bigoplus_{i=1}^n E_i$. But it is known that each of the irreducible finite-dimensional represenations of $\mf{sl}(2,\R)$ lifts to a unique representation of $SL(2,\R)$. Therefore, we may lift $\phi$ by lifting in each subspace $E_i$, and taking the corresponding direct sum of representations.
\end{proof}





\subsubsection{$SL(2,\R)$ and Hyperbolic Geometry}

Recall that the group $PSL(2,\R) = SL(2,\R) / \set{\pm \id}$ is isomorphic to $\Isom(\H^2)$, which gives a canonical action of $SL(2,\R)$ on $\H^2$.

\begin{lemma}
\label{lem:vol-entropy}
There exists $C > 0$ with the following property: If $0 < \ve < 1$ and $S \subset SL(2,\R)$ has $S \subset B(e,R)$, then the minimal number of $\ve$-balls in $SL(2,\R)$ required to cover $S$ is less than $C\ve^{-3}e^{2CR}$.
\end{lemma}

\begin{proof}
Notice that $\mu(B(x,\ve)) = \mu(B(y,\ve)) > C_1\ve^3$ for the standard hyperbolic measure $\mu$ and all $x,y \in \H^2$, since $\ve$ is sufficiently small. Notice also that $\mu(B(x_0,R)) \le C_2e^{h_G R}$, where $h_G$ is larger the exponential volume growth rate for balls in $G$. Therefore, one has at most $C_2e^{2h_GR}/\mu(B(e,\ve/2)) = C_1^{-1}C_2(\ve/2)^{-3}e^{2h_GR}$ disjoint balls in $B(e,R)$. Any such set which is maximally chosen will also cover $B(e,R)$ when taking $\ve$-balls, so the result holds.
\end{proof}

The following Lemma gives estimates on distances of the horocycle flow on $SL(2,\R)$:

\begin{lemma}
\label{lem:dist-computation}
There exists $C > 0$, $t_0 > 1$ such that $d(e,\exp(tU)) \le C \log t$ if $t \ge t_0$.
\end{lemma}

\begin{proof}
Write $\exp(tU) = \begin{pmatrix} 1 & t \\ 0 & 1 \end{pmatrix}$ as $\exp(tU) = k_1ak_2$, where $k_1,k_2 \in SO(2,\R)$ are rotation matrices, and $a = \exp(sX)$ is a diagonal matrix. Observe that $d(e,\exp(tU)) \le d(e,k_1ak_2) \le d(e,k_1) + d(k_1,k_1a) + d(k_1a,k_1ak_2) \le 2D + d(e,a)$, since the metric on $SL(2,\R)$ is right-invariant, and left-invariant under $SO(2,\R)$. But $d(e,a) = s$, and we may compute $s$ by finding the eigenvalues of $\exp(tU)\exp(tU)^T = k_1ak_2k_2^Ta^Tk_1^T = k_1a^2k_1^{-1}$. Notice that:

\[ \exp(tU)\exp(tU)^T = \begin{pmatrix}1+t^2 & t \\ t & 1 \end{pmatrix} \]

which has top eigenvalue $\frac{1}{2} \left(t^2+\sqrt{t^2+4} t+2\right)$. Therefore, the distance from $e$ to $\exp(sX)$ is $\frac{1}{2}\log\left(\frac{1}{2} \left(t^2+\sqrt{t^2+4} t+2\right)\right) \le 2\log t$. Therefore, by choosing $t_0$ large enough we get that $d(\exp(tU),e) \le 2\log t + 2D \le C\log t$.
\end{proof}

\subsubsection{Presentation of Group Elements}

Let $G$ be a Lie group, $\mf g = \Lie(G)$, and $\mf g = \mf e_1 \oplus \mf e_2 \oplus \dots \oplus \mf e_n$ be a vector subspace decomposition of $\mf g$. We do not require that the subspaces $\mf e_1$ are subalgebras or that they commute with one another. The following is an easy adaptation of the classical lemma that $\exp : \mf g \to G$ is a local diffeomorphism at 0.

\begin{lemma}
\label{lem:presentation}
If $g \in G$ is sufficiently close to $e \in G$, then there exists unique $X_i \in \mf e_i$ close to 0 such that $g = \exp(X_1)\exp(X_2) \dots \exp(X_n)$
\end{lemma}

\begin{proof}
Let $\varphi : \mf g \to G$ be the map defined via $\varphi(X) = \exp(X_1)\exp(X_2)\dots \exp(X_n)$, where $X = \sum X_i$ and $X_i \in \mf e_i$. One can easily check that $\varphi'(0) = \id$, and hence $\varphi$ has a local inverse at $e = \varphi(0)$ by the inverse function theorem.
\end{proof}

\subsection{Properties of unipotent flows}\label{uni:flows}

The following definition is important for describing the orbit growth of a unipotent flow (see \cite{KanigowskiVinhageWei}).
\begin{definition}
	\label{def:chainbasis}	
Let $\mf g$ be a Lie algebra and $U \in \mf g$ be a unipotent element. A {\it chain in $\mf g$ with respect to $U$ of depth $m$} is a linearly independent set $\set{ X_i : 0 \le j \le m}$  such that $X_0$ is in the centralizer of $U$ and:
	
	\[ \ad_U(X_i) = X_{i-1} \mbox{ for all }1 \le i \le m.\]

A chain basis of $\mf g$ with respect to $U$ is a basis of chains.  The sequence of depths $(m_1,\dots,m_n)$ of  chains is called the {\it chain structure} of $U$. We will denote the chain basis by $\{X_i^{1}\}_{i=0}^{m_1}, \{X_i^2\}_{i=0}^{m_2},\dots, \set{X_i^n}_{i=0}^{m_n}$. We will often denote chains using the notation:

\[ X_n \mapsto X_{n-1} \mapsto \dots \mapsto X_1 \mapsto X_0 \]
\end{definition}

While every unipotent element $U$ has a chain basis, we will use special structures associated to semisimple groups to construct a canonical one. In particular, notice that the weight spaces for the representations of the $\mf{sl}(2,\R)$ triple can be taken as the chain basis by \eqref{eqn:Uaction}. We reindex them replacing the index $n-2k$ by $i$ for convenience. Therefore, the basis element $X_i^j$ is an eigenvector for $\ad_X$ with eigenvalue $m_j - 2i$. The elements $V \mapsto X \mapsto U$ may be taken as a chain, so there is always at least one chain of depth 2. We call this the {\it Jacobson-Morozov chain}. This implies that any unipotent flow in a semisimple homogeneous space has $GR(U) \ge 3$. Call any chain of depth 0 a {\it trivial chain}. Note that trivial chains span trivial subrepresentations of $\ad$. 




\begin{lemma}\label{th2} Let $\phi_t(g\Gamma)=\exp(tU)g\Gamma$ act on $G\slash \Gamma$ ergodically. The following are equivalent:
\begin{enumerate}
\item $GR(U)=3$.
\item The only nontrivial subrepresentation of $\ad$ is the Jacobson-Morozov representation
\item $\dim G-\dim C(X)\leq 3$
\item $\mf g \cong \mf{sl}(2,\R)\oplus \mf g'$,
$\Gamma$ is irreducible and under this isomorphism,  $U = \left(\begin{pmatrix} 0 & 1 \\ 0 & 0 \end{pmatrix}, \mbf{0}\right) \in \mf{sl}(2,\R)\oplus \mf g'$.
\end{enumerate}
\end{lemma}

\begin{proof}
We show that 1. $\implies$ 2. $\implies$ 3. and 3. $\implies$ 2. $\implies$ 4. $\implies$ 1. That 1. $\implies$ 2. is a direct consequence of the definition of $GR$ and the fact that the Jacobson-Morozov representation has depth 2. Now assume 2. As discussed in Section \ref{sec:sl2}, the trivial chains span trivial subrepresentations for $\ad|_{\mf{sl}(2,\R)}$ and therefore $\ad_X$ also acts trivially. That is, the remaining chain basis elements are in $C(X)$.

Now suppose 3. We claim that this implies $\dim G - \dim C(X) = 2$.  Notice that we have the lower inequality since $U,V \in\mf g$, but neither $U$ nor $V$ commute with $X$. We must therefore rule out the case of 3. This implies that there exists exactly one more linearly independent element which fails to commute with $X$. But by considering each representation $\pi_n$ as described in Section \ref{sec:sl2}, we see that there are always an even number of linearly independent elements which fail to commute with $X$ in each chain basis. In particular, we have 2.

Now assume 2. We claim that in this case $\mf{sl}(2,\R)$ is an ideal in $\mf g$. Indeed, all basis elements which are not from the Jacobson-Morozov representation act trivially on $V$, $X$ and $U$ (since the Lie bracket is anti-commutative), and $V$, $X$ and $U$ act on each other by the standard $\mf{sl}(2,\R)$ relations. Since $\mf g$ is semisimple and $\mf{sl}(2,\R)$ is an ideal, there exists a complementary subalgebra $\mf g'$. Since the flow must be ergodic, $\Gamma$ must be irreducible. That is, we have 4.

One can see 4. $\implies$ 1. by direct computation.
\end{proof}

\begin{proof}[Proof of Corollary \ref{cor2}]
Let $G$ be a semisimple Lie group of dimension at least 4. Then $\mf g = \Lie(G) \not\cong \mf{sl}(2,\R)$. If $\mf g \not\cong \mf{sl}(2,\R) \oplus \mf g'$, by Lemma \ref{th2}, any unipotent $U$ has $GR(U) > 3$. In particular, any ergodic unipotent flow on the quotient of a simple Lie group other than covers of $SL(2,\R)$ are non-standard (this proves (ii)). If $\mf g = \mf{sl}(2,\R) \oplus \mf g'$ , then $\mf g'$ is also semisimple (since $\mf g$ is semisimple). Therefore, each simple factor has a Cartan subalgebra, with associated roots, and in particular, has unipotent elements of the root spaces. Take $U' = U + \sum U_{\alpha_i}$, where $U_{\alpha_i}$ is an element from a root space in each simple factor of $\mf g'$. By the Howe-Moore ergodicity theorem, the action of $U'$ is ergodic, and $GR(U') > 3$ since each $U_{\alpha_i}$ has its own $\mf{sl}(2,\R)$ triple in its semisimple factor (the $\mf{sl}(2,\R)$ triple for $\mf g$ will be the sum of the elements from each factor). Therefore, we have produced an ergodic unipotent flow which is non-standard (this proves (i)).

Finally, assume that $\rank_\R(G) \ge 2$, and let $\mf a$ denote a Cartan subalgebra of $\mf g$. We have shown $G$ carries an ergodic unipotent flow, so if there is one which is not ergodic, we produce two flows which are not Kakutani equivalent. Therefore, we need to produce elements $U$ and $U'$ such that $GR(U) \not= GR(U')$. If $G$ is not simple, we take some $U$ in one factor and $V$ in another, and set $U' = U+V$. This clearly yields two flows with different values for $GR$. If $G$ is simple and has rank at least 2, we may decompose $\mf g = \mf g_0 \oplus \bigoplus_{\alpha \in \Delta} \mf g_\alpha$ into a root space decomposition for some $\mf a \subset \mf g_0 \subset \mf g$. Here, $\mf g_0$ is the centralizer of $\mf a$, a split Cartan subalgebra of $\mf g$, and $\Delta$ is a set of $\R$-valued functionals on $\mf a$ such that if $X \in \mf a$ and $Y \in \mf g_\alpha$, then $[X,Y] = \alpha(X)Y$. Now take any two roots $\alpha_1, \alpha_2 \in \Delta$ which are non-proportional, and let $U_{\alpha_i} \in \mf g_{\alpha_i}$ be elements of the root spaces. We set $U = U_{\alpha_1}$ and $U' = U_{\alpha_1} + U_{\alpha_2}$.

Let $X_m \mapsto \dots \mapsto X_1 \mapsto X_0 \mapsto 0$ be any chain for $U'$. We will show that from this we may produce chains for $U$, hence we may choose a chain for $U$ as chains smaller than that of $U'$. If we show that there is at least one chain that is broken up into two smaller ones, then we get strict inequality, as desired. Write $X_i$ as a sum of root spaces $X_i = \sum_{\beta} X_1^\beta$, then $\ad_{U_{\alpha_1}}(X_i) = \sum_{\beta} \ad_{U_{\alpha_1}}(X_i^\beta) =\sum_{\beta}X_i^{\beta + \alpha_1}$. We may choose $\alpha_1$ and $\alpha_2$ as simple roots, and since all roots are integral linear combinations of the simple roots, every $\beta \in \Delta$ has uniquely determined integer coefficients for $\alpha_1$ and $\alpha_2$, call them $l_1(\beta)$ and $l_2(\beta)$. Let $\mf g_k = \displaystyle\bigoplus_{\beta: l_1(\beta)+l_2(\beta)=k} \mf g_{\beta}$. Then $\ad_U$ and $\ad_{U'}$ both map $\mf g_k$ to $\mf g_{k+1}$. Hence we may choose chain bases for both so that $X_i^j \in \mf g_k$ for some $k$. Notice that any chain for $U'$ which starts at $\mf g_k$ can be decomposed into chains for $U$, by starting a new chain for $U$ at $\ad_{U_{\alpha_2}}(X_i^j)$, if $X_i^j$ we a chain basis element for $U'$). In particular, the Jordan blocks for $U$ are shorter than the corresponding ones for $U'$. Finally, we need to find at least one chain for $U'$ which is broken into shorter ones for $U$. Consider the $\mf{sl}(2,\R)$-triple for $U_{\alpha_2}$, giving $V_{\alpha_2} \in \mf g_{-\alpha_2}$ and $X_{\alpha_2} \in \mf g_0$. Since $\alpha_1$ and $\alpha_2$ are both simple, $\alpha_1-\alpha_2\not\in\Delta$ (the integral coefficients for the simple roots are either all positive or all negative). In particular, $\ad_U(V_{\alpha_2}) = 0$. But $\ad_{U'}(V_{\alpha_2}) = X_{\alpha_2}$, showing that at least one chain is shorter for $U$ than for $U'$. This proves (iii).

If $G \cong SL(d,\R)$, an explicit formula for $GR(U)$ can be found. If $U_l \in \mf{sl}(d,\R)$ has one Jordan block of size $l$, then:

\[GR(U_l)=\frac{1}{6}l(4l+1)(l-1) + l(d-l)(l-1)\]
Since it acts via the representation $\pi_{l-1}$ on the off-diagonal blocks, with the main term coming from Corollary 1.13 of \cite{KanigowskiVinhageWei}. One easily confirms that these are distinct numbers for $l = 2,\dots,d$ by computing $GR(U_{l+1})-GR(U_l) = l(2d-l)$, giving $d-1$ flows with different Kakutani invariant. This proves (iv).

\end{proof}

\begin{remark}
\label{rem:cubic-diff}
In fact, the leading term of the general formula for $GR(U)$ obtained in \cite{KanigowskiVinhageWei} is cubic, and $GR(U_{\alpha_2})$ (in the notation above) grows linearly in $d$. We expect most numbers interpolating the cubic and linear growth to be possible by taking more involved Jordan block structures, and therefore expect the number of pairwise non-Kakutani equivalent flows on $SL(d,\R)$ to grow like $d^3$.
\end{remark}

\subsection{Minimal Growth Rates}

\begin{lemma}
\label{lem:GRge5}
If $G$ is a semisimple group, and $U$ is a unipotent element such that $GR(U) > 3$, then $GR(U) \ge 5$.
\end{lemma}

\begin{proof}
Let $\set{V,X,U} \subset \mf g$ be an $\mf{sl}(2,\R)$ triple for $U$. For contradiction, assume that $GR(U) = 4$. Since there is always one Jordan block of $\ad_U$ of size 3 (corresponding to $V \to X \to U$), there must be only one other Jordan block of length 2, and all other Jordan blocks are trivial. In particular, the eigenvalues of $\ad_X$ are $\pm 2$, $\pm 1$ and 0, with each of the nonzero eigenspaces being simple.

Since $\ad_X$ acts $\R$-semisimply on $\mf g$, and is contained in a split Cartan subalgebra $\mf a$. Since each of the eigenspaces of $\ad_X$ are simple, they must be roots of $\mf g$.  We may therefore choose a Cartan involution $\theta : \mf g \to \mf g$ such that $\mf g = \mf k \oplus \mf p$, with $\mf k = \Fix(\theta)$ a maximal compact subgroup of $\mf g$ and $\mf p$ is a vector subspace of $\mf g$ containing $\mf a$ such that every element is $\ad$-semisimple (the $-1$ eigenspace of $\theta$). Furthermore, if $\Delta_+$ is a set of positive roots for $\mf a$, for each $\alpha \in \Delta_+$, there exists $X_\alpha$ and $X_{-\alpha} := \theta(X_\alpha)$ such that $X_\alpha + X_{-\alpha}$ generates $\mf p$. In particular, $X_\alpha$ is not fixed by $\theta$, and the $\mf{sl}(2,\R)$-triple is invariant under $\theta$. Let $V_1 \to V_2$ be the Jordan block for $\ad_X$ of length 2. Then $V_1$ and $V_2$ generate root spaces, and since $\ad_U(V_1) = V_2$, we know that if $\alpha$ is the root corresponding to $U$ and $\beta$ is the root for $V_1$, then $\alpha+\beta$ is the root for $V_2$. However, $\theta(V_2) \to \theta(V_1)$ is also a nontrivial block, and since there are only two blocks for $\ad_U$, it must coincide with $V_1 \to V_2$. Therefore, $\alpha + \beta = -\beta$ and $\alpha = -2\beta$. Any semisimple group which has resonance of the form $\alpha = -2\beta$ must also have the corresponding $\beta$-subspace with dimension greater than 2 (this can be checked case-by-case for the non-split real forms). This contradicts the simplicity of the root spaces for $\alpha$ and $\beta$, and we arrive at a contradiction.
\end{proof}

\begin{remark}
This lower bound is sharp, and can appear in two ways. Given a matrix group or algebra, we let $E_{ij}$ denote the matrix with $1$ in the $(i,j)$ position and 0 in every other entry. First, if one takes the flow generated by $E_{12}$ in $SL(3,\R) / \Gamma$, one can see that other than the $\mf{sl}(2,\R)$-triple, there are exactly two nontrivial Jordan blocks for the action: $E_{23} \to E_{13} \to 0$ and $E_{32} \to E_{31} \to 0$. In this case the Cartan involution does not reverse the direction of the Jordan block, and in fact these Jordan blocks are permuted by the Cartan involution. If the Jordan blocks are not fixed by this involution, then there must be more than one.

The other example is that of $SU(2,1)$, with Lie algebra

\[ \mf{su}(2,1) \cong \set{ \begin{pmatrix} z & it & w_1 \\ is & -\bar{z} & w_2 \\ -\bar{w}_2 & -\bar{w}_1 & -2\Im(z) \end{pmatrix} : z,w_1,w_2 \in \C, \; t,s \in \R} \]

Then consider the flow generated by $iE_{12}$, with corresponding Cartan subalgebra $\mf a = \set{\diag(t,-t,0) : t \in \R}$. Notice that there is an $\mf{sl}(2,\R)$ triple which is the usual $\mf{sl}(2,\R)$ triple in the top left $2 \times 2$ block, with the unipotent elements scaled by $i$. The diagonal matrix $\diag(i\theta,i\theta,-2i\theta)$ commutes with the flow, and is a trivial block of size one. Since $w_2$ can take values in $\C$, there are two blocks coming from this root space (which is two dimensional):

\[ E_{23} - E_{31} \to i(E_{13} + E_{32}) \to 0 \qquad iE_{23} + iE_{31} \to -E_{13} + E_{32} \to 0 \]

Notice that in this case, while there is still more than one such Jordan block, the Cartan involution fixes each one, so in the proof we must use the fact that every algebra for which $\alpha$ and $2\alpha$ are roots has $\dim(\mf g_\alpha) > 1$. In both examples, the growth rate is equal to 5.
\end{remark}

\section{Some additional definitions}
\label{sec:more-defs}

Let $G$ be a semisimple Lie group and $\Gamma \subset G$ be a finite volume lattice as considered in Section \ref{sec:prelims}, with projection $\pi : G \to G / \Gamma$. We fix a fundamental domain $F \subset G$ for $G/ \Gamma$. Then every point of $G / \Gamma$ has at least one lift to $F$, and there is a unique lift on an open, dense subset of $G / \Gamma$ (corresponding to the interior of $F$).


Notice that $\pi : F \to G / \Gamma$ is a measurable isomorphism, so we can think of the left action $G$ on $G / \Gamma$ equivalently as an action on $F$. In particular, if $x,y \in G/ \Gamma$, then they are cosets $x = \widetilde{x}\Gamma$ and $y = \widetilde{y}\Gamma$ for some unique (except for points on the boundary of $F$) $\tilde{x},\tilde{y} \in F$. Then $gx = y$ is equivalent to $g\tilde{x} = \tilde{y}\gamma$ for some $\gamma \in \Gamma$. This allows us to consider the homogeneous flow $(\phi_t)$ on $F$ instead on $G / \Gamma$, as we shall in Section \ref{sec:technical}. which preserves the Haar measure $\mu_F$. We may lift the metric on $G / \Gamma$ to $F / \sim$ by setting $d_F(\tilde{x},\tilde{y}):=\inf_{\gamma\in \Gamma}d_G(\tilde{x},\tilde{y}\gamma)$, where $F / \sim$ is the topological quotient space of $F$ by the usual relation.

For a point $y\in G / \Gamma$ let ${\rm inj}(y)$ denote the injectivity radius of $y$, i.e.
$$
{\rm inj}(y):=\sup \{r\geq 0 : B_G(y,r)\cap B_G(y,r)\gamma = \emptyset \mbox{ for all }\gamma \not= e\}.
$$
For a set $K\subset G / \Gamma$ let ${\rm inj}(K)=\inf_{y\in K}{\rm inj}(y)$.
We have the following classical lemma which we state here for reference.
\begin{lemma}\label{lem:sys}
For every $\epsilon>0$ there exists a compact set $K_\epsilon\subset F$,
 $\mu(K_\epsilon)>1-\epsilon$ and such that
$$
\kappa(\epsilon):={\rm inj}(K_\epsilon)=\inf_{\gamma\in \Gamma\setminus\{e\}}\inf_{z\in K_\epsilon}d_G(z\gamma z^{-1},e)>0.
$$
\end{lemma}


\subsection{Kakutani-Bowen Balls}
\label{sec:kak-bowen-balls}

Given a unipotent element $U \in \mf g$, we may use the results of Sections \ref{sec:sl2} and \ref{uni:flows} to obtain a basis $\set{V,X,U} \cup \set{X_i^j : i = 1,\dots,m_j, j = 1,\dots,n}$, where $\set{V,X,U}$ generates a subalgebra of $\mf g$ and satisfy the standard relations for the $\mf{sl}(2,\R)$-triple, and the remaining elements are chains for $U$. We may therefore apply Lemma \ref{lem:presentation} to write elements of $G$ sufficiently close to $e$ as

\begin{equation}
\label{eq:coefficients}
g =  \exp(a_V(g)V) \exp(a_X(g)X)\exp\left( a_U(g)U + \sum a_{ij}(g) X_i^j\right)
\end{equation}

Let $\tau(g) = \exp\left( a_U(g)U + \sum a_{ij}(g) X_i^j\right)$ be the {\it standard chain component} of $g$. The following definition combines dynamical and algebraic features, which is critical to our analysis of the Kakutani balls $B_R(x,\ve,\mc P)$ (see Theorem \ref{prop:longblock}).

\begin{definition}[Kakutani-Bowen Balls]\label{def:kakball}
For $\epsilon>0$, $R>0$ let
$$
\begin{array}{c}
\Bow(R,\epsilon,e):=\{g\in G:d_G(\exp(rU)g\exp(-rU),e)<\epsilon, \text{ for every } r\in [0,R]\}.
\end{array}$$

be the {\it Bowen ball} of $e \in G$ for $U$. If $x,y\in G / \Gamma$ we say that
$x\in \Kak(R,\epsilon,y)$
if and only if 
$x = gy$ and
\begin{enumerate}[(a)]
\item \label{a-condition}$\abs{a_V(g)} < \ve / R$
\item $\abs{a_X(g)} < \ve$
\item $\tau(g) \in \Bow(R,\ve,e)$
\end{enumerate}
where $a_V(\cdot)$, $a_X(\cdot)$ and $\tau(\cdot)$ are as in \eqref{eq:coefficients}.
\end{definition}

We will often consider Kakutani balls as subsets of $F$, since points in $F$ are in one-to-one correspondence with points in a compact subset $K\subset G /\Gamma$ (except for those on the boundary).  Furthermore, if $\ve$ is sufficiently small, depending only on $\inj(K)$, if $\tilde{y}$ is any point of $G$ such that $\tilde{y}\Gamma = y$, $\Kak(R,\ve,y)$ lifts uniquely to a neighborhood of $\tilde{y}$. The definition of $\Kak(R,\epsilon,y)$ has the following explanation. We will see that points that differ in the direction $X_i^k$ will see polynomial divergence in the direction $X_j^k$ with degree $i-j$ for $j < i$. Since $\set{V,X,U}$ is taken as a chain for $U$, points that differ in $V$ direction split with quadratic speed in the direction of $U$ and with linear speed in the direction of $X$. Such points can be easily matched with the $f_t$-metric (even though they can not by $\bar{d}$-Bowen metric) as long as we don't see divergence in $X$ direction, since we are allowed to correct by the flow. This is the reason why in the definition of Kak we take $|a_V|\leq \frac{\epsilon}{R}$. Similarly, points differences in $X$ direction yield linear divergence, but  only in the $U$ direction and hence the control required on the $X$ coefficient does not grow.

Since the condition for being in a Kakutani ball does not tell us the direction of divergence, we make the following additional definitions which allow us to distinguish points whose divergence is first seen in the $\mf{sl}(2,\R)$-triple components, and those which see divergence in the other chains first.

Fix $\delta > 0$, and set
\be\label{emj}
\Kak^{1,\delta}(R,\epsilon,y):=\Kak(R,\epsilon,y)\cap\set{x\in  G / \Gamma: |a_V(g)|<R^{-(1+10\delta)}\text{ in Definition \ref{def:kakball}} }
\ee
and
\be\label{emj2}
\Kak^{2,\delta}(R,\epsilon,y):= \Kak(R,\epsilon,y)\setminus \Kak^{1,\delta}(R,\epsilon,y).
\ee

For sufficiently large $R$, $\Kak^{1,\delta}$ are exactly points in which the divergence is guaranteed to be seen first in the other chains (due to the increased control on $a_V$). $\Kak^{2,\delta}$ is therefore the points in which we are guaranteed to see some divergence in the $X$ direction, since in this case $\abs{a_V} \in [R^{-(1+10\delta)},\ve/R]$.

\begin{definition}[Splitting time]\label{def:splittime} For $x,y\in G / \Gamma$ define the {\it splitting time} of $x,y$
\be\label{eq:splitnew}
S(x,y,\epsilon):=\sup \{R\geq 0\;:\;x\in \Kak(R,\epsilon,y)\}.
\ee
\end{definition}

The following observation is a straightforward consequence of continuity of $(\phi_t)$: there exists a function $f : \R\to \R$ such that $f(m)\to +\infty$ as $m\to +\infty$ and
\be \label{lem:contphi}
S(x,y,\epsilon)>f(m) \text{ as } d_{G / \Gamma}(x,y)<m^{-1}.
\ee
We also have the following general definition which establishes a useful notation when dealing with matching of $x,y$. A priori, given a matching between $x$ and $y$, the points may have long periods of matching, diverge for a small amount of time, and realign to have another long period of matching. This is exactly what happens for the classical horocycle flow. The following definition identifies that maximal interval on which the matching could be extended before seeing divergence and waiting for another realignment.

\begin{definition}\label{def:matchsplit} Fix a partition $\cP$ of $G / \Gamma$ and $x,y\in G / \Gamma$ which are $(\eta, \cP)$-matchable (with matching function $h$). For $u\in A(x,y)$ denote $x_u=\phi_u x\in G / \Gamma$, $y_u=\phi_{h(u)}y\in G / \Gamma$ and let for $\epsilon>0$
$$
S(u,\epsilon)=S(x_u,y_u,\epsilon).
$$
\end{definition}

\subsection{Summary of notations}
\begin{center}
\begin{tabular}{|l|l|}
\hline
$G$ & A fixed semisimple linear Lie group \\
\hline
$\mf g$ & The Lie algebra of $G$ \\
\hline
$\Gamma$ & A lattice in $G$ \\
\hline
$x,y,z$ & Points of $G / \Gamma$ \\
\hline
$\tilde{x},\tilde{y},\tilde{z}$ & Lifts of $x,y,z$ to $G$ \\
\hline
$g,h$ & Elements of $G$ \\
\hline
$\inj(G/\Gamma)$ & The largest number such that if $d_{G/\Gamma}(x,y) < \inj(G/\Gamma)$, \\ & $x = gy$ for a unique $g \in B_G(e,\inj(G/\Gamma))$ \\
\hline
$\set{V,X,U}$ & Fixed generators of a subalgebra isomorphic to $\mf{sl}(2,\R)$ \\
\hline
$\phi_t$ & The left translation action by $\exp(tU)$ \\
\hline
$GR(U)$ & The polynomial slow entropy of $\phi_t$ (ie, the growth rate \\ & for the number of Bowen balls to cover $M$) \\
\hline
$\set{X_i^j}_{i=1}^{m_j}$, $j = 1,\dots,n$ & Vectors generating $\mf g$ together with $\set{V,X,U}$, having \\
 & certain relations with $\set{V,X,U}$ (see Section \ref{sec:sl2}) \\
\hline
$\Bow(R,\ve,y)$ & The Bowen ball around $y$ of radius $\ve$ up to time $R$ \\
\hline
$B_R(y,\ve,\mc P)$ & The Kakutani ball around $y$, ie the set of all points\\
 &  $x$ which are  $(\ve,\mc P)$-matchable with $y$ \\
\hline
$\Kak(R,\ve,y)$ & Intuitively, points which stay $\ve$-close to $y$ after lifting to \\
 & $G$, allowing correction of $x$ by the flow (see Section \ref{sec:kak-bowen-balls})\\
\hline
$\Kak^{1,\delta}(R,\ve,y)$ & A set of points in $\Kak(R,\ve,y)$ which see non-orbit \\ & divergence in directions other than $X$ first \\
\hline
$\Kak^{2,\delta}(R,\ve,y)$ & The remaining points of $\Kak(R,\ve,y)$ \\
\hline
$S(x,y,\ve)$ & The first time $x,y$ split and cannot be made close  \\ & by applying $\phi_t$ to $y$ in the universal cover \\
\hline
$S(u,\ve)$ & With a fixed matching of $x$ and $y$, $S(x_u,y_u,\ve)$, \\ & where $x_u$ and $y_u$  is the matching at time $u$\\

\hline
\end{tabular}
\end{center}

\section{Orbit divergence estimates}
\label{sec:orb}

In this section we state results on orbit divergence for unipotent flows. These results play an important role in the proofs of Theorem \ref{th1} and \ref{prop:longblock}.

 We recall that following lemmas, which will be used in the proof:

\begin{lemma}
	\label{lem:polynomial-coefficients}
	Let $p(t) = \sum_{k=0}^d a_kt^k$ be a polynomial of degree $d$. There exists $C(d)$ such that if $\abs{p(t)} < \epsilon$ for all $t \in [0,T]$, then $\abs{a_k} < C(d)T^{-k}\epsilon$ for all $k = 0,\dots,d$. Conversely, if $\abs{a_k} < C(d)^{-1}T^{-k}\epsilon$ for all $k$, then $\abs{p(t)} < \epsilon$ for all $t \in [0,T]$.
\end{lemma}

Let $g = \exp\left(a_U(g)U + \sum_{j=1}^n\sum_{i=0}^{m_j} a_{ij}(g)X_i^j\right)$.
The following formulas are important for computing divergence rates:

\begin{multline}
\label{eq:conj-formula}
\exp(sX)\exp(tU)g \exp(-tU)\exp(-sX) \\= \exp\left( e^{2s}a_U(g)U + \sum_{j=1}^n \sum_{i=0}^{m_j}\left(\sum_{k=0}^{m_j-i} e^{(m_j-2i)s}\frac{t^k}{k!}a_{(k+i)j}(g)\right)X_i^j\right)
\end{multline}

\begin{multline}
\label{eq:conj-formula2}
\exp(tU)\exp(sX) g\exp(-sX)\exp(-tU) \\= \exp\left( e^{2s}a_U(g)U + \sum_{j=1}^n \sum_{i=0}^{m_j}\left(\sum_{k=0}^{m_j-i} e^{(m_j-2(k+i))s}\frac{t^k}{k!}a_{(k+i)j}(g)\right)X_i^j\right)
\end{multline}

Equations \eqref{eq:conj-formula} and \eqref{eq:conj-formula2} follow from Lemma \ref{lem:adjoint} and the choice of chain basis made in section \ref{sec:sl2}. The proof of the following lemma is almost identical to that of \cite[Proposition 3.7]{KanigowskiVinhageWei}, so we provide only a sketch.

\begin{lemma}\label{lem:B_Gdiv} There exists $\epsilon_0$ such that for every $y\in G / \Gamma$ and every $\epsilon\in(0, \min(\ve_0,{\rm inj}(y)/3))$, we have
\be\label{eq:eq2}
C(\ve_0)R^{-GR(U)+2}\geq \mu(\Kak(R,\epsilon,y))\geq C(\ve_0)R^{-GR(U)+2}.
\ee
\end{lemma}
\begin{proof}[Sketch of Proof]

Since $\ve < \inj(y)/3$, 
it follows that the projection $\pi : G \to G / \Gamma$ is injective on $\Kak(R,\ve,y) \subset B(y,3\ve)$.
Write $x = gy$, and note that the coefficients of $X_i^k$ for $\log(\exp(tU)\tau(g)\exp(-tU))$ are all polynomials of degree $i$ by \eqref{eq:conj-formula}. Therefore, by Lemma \ref{lem:polynomial-coefficients} and \eqref{eq:conj-formula} with $s = 0$, a sufficient condition for $\tau(g)$ is that $\abs{a_{ik}} \le C'(\ve) R^{-i}$, and a necessary one is that $\abs{a_{ik}} \le C'(\ve)^{-1}R^{-i}$ (by shrinking $\ve_0$ if necessary to absorb the constant $C(d)$).  Let $L : U \to \R V \oplus \R X \oplus \left(\R U \oplus \displaystyle\bigoplus_{i,k} \R X_i^k\right)$ be the inverse function provided by Lemma \ref{lem:presentation}. Then we have shown that

\begin{eqnarray*} [-\ve/R,\ve/R] \times [-\ve,\ve] \times \left( [-\ve,\ve] \times \prod_{i,k} \left[-\frac{\ve}{C'(\ve)R^i},\frac{\ve}{C'(\ve)R^i}\right] \right) & \subset & L(\Kak(R,\ve,y)) \\
 \subset  [-\ve/R,\ve/R] \times [-\ve,\ve] \times \left( [-\ve,\ve] \times \prod_{i,k} \left[\frac{-C'(\ve)\ve}{R^i},\frac{C'(\ve)\ve}{R^i}\right] \right)
\end{eqnarray*}

Notice that the hypercubes which contain and are contained in $L(\Kak(R,\ve,y))$ decay with the rate prescribed. Since the Jacobian of $L$ is bounded above and below in a neighborhood of $e$ in $G$, we get the desired decay rate.

\end{proof}

In the study of slow entropy, ie the covering rate for $G / \Gamma$ via Bowen balls, the result analogous to Lemma \ref{lem:B_Gdiv} is sufficient to estimate the number of Bowen balls to cover the space. However, Kakutani balls have a more complicated behavior, since we only insist that the points are close for a large proportion of times. The remaining lemmas help to show that if points stay together for a certain interval, then that amount of time can be quantified, and that each such interval has a long interval afterwards in which the points diverge, but in a controlled way as to avoid recurrence.

The following lemma allows us to explictly describe an optimal matching function when $x$ and $y$ are sufficiently close.

\begin{lemma}\label{lem:slren}There exists $\epsilon_1>0$ such that if we let  $h=\exp(a_VV)\exp(a_XX)\in G$ and $\psi(t):=te^{a_X}/(e^{-a_X}-a_Ve^{a_X}t)$ with $|a_X|<\epsilon_1$, then for every $|t|\in [0,\epsilon_1 a_V^{-1}]$, we have
\begin{equation}\label{eq:sl2spl}
\exp(\psi(t)U)h\exp(-tU)=\exp(\alpha_tV)\exp(\beta_tX),
\end{equation}
with $|\alpha_t|\leq 2|a_V|$ and $|\beta_t|\leq 2(|a_X|+|a_V||t|)$.
\end{lemma}
\begin{proof}
By Lemma \ref{lem:sl2-lift}, we may make computations in $SL(2,\R)$ and conclude the relevant relations in $G$. Abusing the notation in this proof slightly, we let $V,X, U \in \mf sl(2,\R)$ denote the generators of the opposite horocycle flow, geodesic flow, and horocycle flow respectively.
By a direction computation and the definition of $\psi(\cdot)$, we have
\begin{equation}
\exp(\psi(t)U)\exp(a_VV)\exp(a_XX)\exp(-tU)=\left(
                                           \begin{array}{cc}
                                             e^{a_X}+\psi(t)a_Ve^{a_X} & 0\\
                                             a_Ve^{a_X} & e^{-a_X}-a_Ve^{a_X}t \\
                                           \end{array}
                                         \right).
\end{equation}

Let $\alpha_t$ and $\beta_t$ be defined so that $$\exp(\alpha_t V)\exp(\beta_t X)=\exp(\psi(t)U)\exp(a_VV)\exp(a_XX)\exp(-tU).$$
Direct computation shows that
\begin{equation}
\begin{aligned}
&\alpha_t e^{\beta_t}=a_Ve^{a_X},\\
&e^{-\beta_t}=e^{-a_X}-a_Ve^{a_X}t.
\end{aligned}
\end{equation}
Thus we have
\begin{equation}
\begin{aligned}
\alpha_t=a_Ve^{a_X}(e^{-a_X}-a_Ve^{a_X}t),\\
\beta_t=-\log(e^{-a_X}-a_Ve^{a_X}t).
\end{aligned}
\end{equation}
Since $|a_X|<\epsilon_1$ and $|t|\in [0,\epsilon_1 a_V^{-1}]$, we have
$$|\alpha_t|<2|a_V|, |\beta_t|\leq2(|a_X|+|a_V||t|).$$
This finishes the proof.
\end{proof}

\begin{remark}\label{rem:inv} We will use the following property of $\psi(\cdot)$ which follows by a direct computation: under the above assumptions if additionally $|a_X|<\epsilon^2$, then for every $t\in [0,\epsilon^{2}a_V^{-1}]$ we have
$$
\psi'(t)\in(1-\epsilon,1+\epsilon).
$$
\end{remark}

\begin{lemma}\label{lem:adu}There exists $\epsilon_1>0$ such that for every $\epsilon\in (0,\epsilon_1]$ and for every $x,y\in G / \Gamma$ if
$x\in \Kak(R,\epsilon^3,y)$, then for every $|L|\in [0,R/3]$ there exists $|\ell|<R$ such that
$$
\phi_{\ell}(x)\in \Kak(R/2,\epsilon,\phi_Ly).
$$

\end{lemma}
\begin{proof}
By Definition \ref{def:kakball}, it follows that we may write
$$
x=\exp(a_VV)\exp(bX)gy
$$
where $|a_V|\leq \epsilon^3R^{-1}$, $|b|\leq \epsilon^3$ and $g \in \Bow(R,\epsilon^3,e)$, with $g$ having no $V$ or $X$ component.
 Therefore, for every $\ell \geq 0$
\be\label{eq:nfgt}
\exp(\ell U)x= \big(\exp(\ell U)\exp(a_VV)\exp(bX)\exp(-L U)\big)\big(\exp(L U)g\exp(-L U)\big)\exp(LU) y
\ee
Moreover by Lemma \ref{lem:slren} if we define $\ell:=\psi(L)$ (by the bound on $a_V$ this is well defined), then $|\ell|\leq R/2$ (see Remark \ref{rem:inv}), and we have
$$
\exp(\ell U)\exp(a_VV)\exp(bX)\exp(-LU)=\exp(b_LV)\exp(c_LX),
$$
where $|b_L|\leq \epsilon R^{-1}$, and $|c_L|<\epsilon$.
Notice also that if $m_\ell=\exp(L U)g\exp(-LU)$ then for every $t\in [0,R-L]$, we have
\be\label{emmel}
d_G(\exp(tU)m_\ell\exp(-tU),e)=d_G(\exp((t+L)U)g\exp(-(t+L)U),e)<\epsilon^3
\ee
the last inequality since $g \in \Bow(R,\epsilon^3,e)$.
Therefore $m_l\in\Bow(R/2,\epsilon^3,e)$ and the $V$ and $X$ coordinates of $m_l$ are zero since the spaces generated by each chain are invariant for $\ad(tU)$ (see Definition \ref{def:kakball}). Therefore

\[ \phi_\ell x = \exp(\ell U) x = \exp(\ell U) \exp(a_VV) \exp(bX) g y = \exp(b_L V)\exp(c_L X)m_\ell \phi_L y \]

which by Definition \ref{def:kakball} and  \eqref{emmel} implies that
$\phi_{\ell} x\in \Kak(R/2,\epsilon, \phi_Ly)$. This finishes the proof.

\end{proof}

The following Lemma quantifies the renormalization phenomenon related to the relation $[X,U] = 2U$. Recall Definition \ref{def:kakball} and equation \eqref{eq:coefficients}.

\begin{lemma}\label{lem:dirs}
Let $g \in \Bow(R,\ve,e)$ be such that $a_V(g) = a_X(g) = 0$.
There exists  $C>0$ (independent of $\ve$) such that for every $\epsilon\in (0,\epsilon_1]$, every $\delta'\in [0,C^{-2})$, every  $R>0$ and every $s\in [0,\frac{1}{2}(1+\delta')\log (R)]$, we have
$$
\exp(-sX)g\exp(sX) \in \Bow(R^{1/2-C\delta'},\epsilon^{1/3},e)
$$
 and $a_V(\exp(-sX)g\exp(sX)) = a_X(\exp(-sX)g\exp(sX))=0$.
\end{lemma}

\begin{proof}
Since $a_V(g) = a_X(g) = 0$, we can write $g=\exp\left(a_U(g)U + \sum_{i,j}a_{ij}(g)X_i^{j}\right)$,  where $\{X_i^0\}_{i=1}^{m_0},\ldots,\{X_i^n\}_{i=1}^{m_n}$ are standard chains, each of which span a finite-dimensional representation of $\mf{sl}(2,\R)$ (see Definition \ref{def:chainbasis}). Note that the claims that $$a_V(\exp(-sX)g\exp(sX)) = a_X(\exp(-sX)g\exp(sX)) = 0$$ follow from the fact that each chain $\set{X_i^j}$ spans a finite-dimensional representation of $\mf{sl}(2,\R)$ and that $U$ is an eigenvector for both $\ad_X$ and $\ad_U$.
Since $g \in \Bow(R,\epsilon,e)$ it follows that for $0\leq t\leq R$, $d_G(\exp(tU)g\exp(-tU),e)\leq\epsilon$. Since $d_G(\exp(tU)g\exp(-tU),e)<\epsilon$ for $0\leq t\leq R$ it follows by Lemma \ref{lem:polynomial-coefficients} applied to the $X_0^j$ terms of \eqref{eq:conj-formula} with $s = 0$ that
\begin{equation}\label{eq:unipotentActResult}
|a_{ij}(g)|\leq\frac{C(d,\ve)\epsilon}{R^{i}}
\end{equation}
for all $j$, where $C(d,\ve)$ is determined by Lemma \ref{lem:polynomial-coefficients} and the norm of $d(\geomlog \of \exp)$ on $B(e,\ve_1)$ (since $d(e,\exp(Y)) = \norm{\geomlog(\exp(Y))}$).



Let $R_s(g) = \exp(-sX) g \exp(sX)$. We need to show that $R_s(g)\in \Bow(R^{1/2-C\delta'},\epsilon^{1/3},e)$. Hence we only need to show that for every
$0\leq t\leq R^{1/2-C\delta'}$,
$$
d_G(\exp(tU)R_s(g)\exp(-tU),e)<\epsilon^{1/3}.
$$

Notice that by \eqref{eq:conj-formula2},

\[ a_{ij}(s,t) = \sum_{k=0}^{m_j-i} \frac{t^k}{k!}e^{-(m_j-2(i+k))s}a_{(i+k)j}(g)\]

denotes the coefficient of $X_i^{j}$ for $\exp(tU)R_s(g)\exp(-tU)$ at time $t$. We will control each coefficient $\alpha_{ijk} = \frac{1}{k!}e^{-(m_j-2(k+i)s)}a_{(k+i)j}(g)$.

Our bound on $a_{ij}(g)$, \eqref{eq:unipotentActResult}, and the condition that $0 \le s \le \frac{1}{2}(1+\delta')\log(R)$ gives:

\[ \abs{\alpha_{ijk}} \le \frac{C(d,\ve_1)\ve}{k!} R^{-\frac{1}{2}(m_j-2(k+i))(1+\delta')-(k+i)} \le C(d,\ve_1)\ve R^{-\frac{1}{2}m_j+(k + i-\frac{1}{2}m_j)\delta'}\]

Therefore, by Lemma \ref{lem:polynomial-coefficients}, since $k+i \le m_j$ and since the maximal power of $t$ by $a_{ij}(s,t)$ is $m_j-i$ (and therefore  $t^k\leq R^{(1/2-C\delta')(m_j-i)}$, we have
\begin{multline*}
 \sup_{t \in [0,R^{1/2-C\delta'}]} a_{ij}(s,t) \le m_j C_1(d,\ve_1)\ve R^{(1/2-C\delta')(m_j-i)-\frac{1}{2}m_j(1-\delta')}=\\m_j C_1(d,\ve_1) R^{-1/2i-C\delta'(m_j-i)+1/2m_j\delta'}
\end{multline*}
Notice that  $-1/2i-C\delta'(m_j-i)+1/2m_j\delta'\leq 0$.

Therefore, if $\ve$ is sufficiently small, we can guarantee that $\sup_{t \in [0,R^{1/2-C\delta'}]} a_{ij}(s,t)$ can be made less than $\ve^{1/3}/C''$ for arbitrary $C''$. So by choosing $\ve$ sufficiently small, we may guarantee $R_s(g) \in \Bow(R^{1/2-C\delta'},\ve^{1/3},e)$.
\end{proof}


\begin{lemma}\label{lem:B_Ggr}There exist constants $C_2,\epsilon_1,R_0>0$ such that for every $\epsilon\leq \epsilon_1$,  every $R\geq R_0$,  $s\leq  \frac{1}{2}\log (R)$,  and for every $g  = \exp(Y)\in \Bow(R,\ve,e)$,

$$
\exp(-sX)g\exp(sX)\in \exp(Y_C)B_{G}(e,C_2\epsilon R^{-\frac{1}{2}}).
$$
for some $Y_C \in C(U)$. Moreover, for $s=\frac{1}{2}\log (R)$, the same holds for some $Y_C \in C(U,X)$.
\end{lemma}
\begin{proof}Recall, (see \eqref{eq:conj-formula} and \eqref{eq:unipotentActResult}) that if $g=\exp\left(\sum_{i,j}a_{ij}(g)X_i^j\right)\in \Bow(R,\epsilon,e)$, then

$$
|a_{ij}(g)|\leq\frac{C(d)i!\epsilon}{R^i},
$$
and
\be\label{edsw}
\exp(-sX)g\exp(sX)=\exp\left(\sum_{j=1}^n\sum_{i=0}^{m_j} e^{-s(m_j-2i)}a_{ij}(g)X_i^j\right),
\ee
 Therefore, if $s\leq \frac{1}{2}\log R$, we get the following bound on the $X_i^j$ coefficient of \\ $\exp(-sX)g\exp(sX)$ (denoting $s=c_s\log (R)$, so that $c_s < \frac{1}{2}$)
$$
|e^{-s(m_j-2i)}a_{ij}(g)|\leq C(d)i!\epsilon R^{-i-c_s(m_j-2i)}.
$$
Moreover, if $m_j-2i\geq 0$, then $-i-c_s(m_j-2i)\leq -i$ and if $m_j-2i<0$, then   $-i-c_s(m_j-2i)\leq -i-\frac{1}{2}(m_j-2i)=-\frac{1}{2}m_j$ and recall that $m_j\geq i$. Therefore, if $i>0$, then for every $j$, the coefficient by $X_i^j$ in \eqref{edsw} is at most $C(d)i!\epsilon R^{-1/2}$. Notice that by Definition \ref{def:chainbasis}, $X^j_0\in C(U)$. Let $Y_C = \sum_{j=0}^n a_{0j}(g)X_0^j$. 
 Pulling $Y_C$ out of the expression for $g$ does not cost much, since all expressions given by Lemma \ref{lem:presentation} are tangent to the identity. In particular, $\exp(A + B) = \exp(A + O(\abs{A}\cdot\abs{B}))\exp(B)$. Therefore,
$$
\exp(-sX)g\exp(sX)\exp(-Z_C)\in B_{G}(e,C_2\epsilon R^{-\frac{1}{2}}).
$$

Moreover, for $s=\frac{1}{2}\log R$ it follows that $-i-\frac{1}{2}(m_j-2i)=-\frac{1}{2}m_j\leq -\frac{1}{2}$, unless $m_j=0$, but chains of length $0$ correspond to vectors in $C(U,X)$ (they must span trivial representations, see Section \ref{sec:sl2})  and are absorbed in $Y_C$. This finishes the proof of the second part.

\end{proof}

Before we state next lemmas, we need the following general lemma about polynomials. We use the following technical tool in the proof:

\begin{lemma}[Brudnyi-Ganzburg inequality \cite{Brudnyi}]
\label{lem:brudnyi}
Let $V \subset \R$ be an interval, and $\omega \subset V$ a measurable subset. Then for any polynomial $p$ of degree at most $d$:

\[ \sup_V \abs{p} \le \left(\dfrac{4\abs{V}}{\abs{\omega}}\right)^d \sup_\omega \abs{p}\]
\end{lemma}

\begin{lemma}\label{lem:pol} Let $p$ be a polynomial with $\deg(p) \le d$. There exists $C_1(d)>0$ such that for every $\eta>0$ if $N>0$ satisfies $|p(N)|\geq \epsilon$ (for $\epsilon\leq \epsilon_0$) and $\abs{p(0)}< C(d)^{-1}\ve$ (recall $C(d)$ is defined in Lemma \ref{lem:polynomial-coefficients}), then
$$
\left|w\in [0,N^{1+\eta}]: |p(w)|\leq 10 \epsilon\right|\leq C_1(d)N^{(1+\eta-\eta/d)}.
$$
\end{lemma}
\begin{proof} Notice that since $p(N)\geq \epsilon$ it follows that for some $t\in[0,N^{1+\eta}]$, we have $|p(t)|\geq C(d)^{-2}\epsilon N^{\eta}$. Indeed, if $|p(t)|< C(d)^{-2}\epsilon N^{\eta}$ on $[0,N^{1+\eta}]$, then by Lemma \ref{lem:polynomial-coefficients} above, the coefficients of $p(\cdot)$ satisfy $|a_k|<C(d)^{-1}\epsilon N^\eta N^{-(1+\eta)k}$. Using this to control coefficients $k > 0$ and the assumption on $p(0)$ for $k = 0$ we may apply the converse of Lemma \ref{lem:polynomial-coefficients}, it follows that $p(N)<\epsilon$.
Then let $\omega:=\{t\in [0,N^{1+\eta}]\;:\; |p(t)|\leq 10\epsilon\}$. By the Brudnyi-Ganzburg inequality, and the above estimate, it follows that
$$
|\omega|\leq \frac{4\sup_{\omega}|p(t)|^{1/d}|V|}{\sup_{V}|p(t)|^{1/d}}.
$$
So
$$
|\omega|\leq 4\frac{(10\epsilon)^{1/d}N^{1+\eta}}{(C(d)^{-2}\epsilon N^{\eta})^{1/d}}\leq 40C(d)^2N^{1+\eta-\eta/d}.
$$
Setting $C_1(d)=40C(d)^2$ finishes the proof.

\end{proof}

The above lemma and the orbit divergence estimate yields the following corollary:

\begin{corollary}\label{corpol}There exists $C_2(d) > 0$ and
$\ve_3>0$ such that for every $\eta>0$ and $\ve_3>\epsilon>0$ if $g\notin \Bow(R,\epsilon,e)$ and $d_G(g,e) < C(d)^{-1}\ve$ (recall $C(d)$ is defined in Lemma \ref{lem:polynomial-coefficients}) then
$$
\left|w\in[0,R^{1+\eta}]\;:\; d_G(\exp(wU)g\exp(-wU),e)<10\epsilon\right|\leq C_2(d)R^{(1+\eta-\eta/d)}.
$$
\end{corollary}
\begin{proof} Recall that if $g=\exp(\sum_{i,j}a_{ij}X_i^j)$, then
$\exp(wU)g\exp(-wU)$ will have the coefficients of $X_i^j$ as polynomials in $w$ by \eqref{eq:conj-formula}. Let $p_{ij}(w)$ denote the polynomial for $X^j_i$ and let $R\geq w\geq 0$ be the largest number such that $g \in \Bow(w,\epsilon,e)$, then analogously to \eqref{eq:unipotentActResult} it follows that there exists $j$ such that
$$
|p_{0j}(w)|\geq \epsilon.
$$
Indeed, this follows from the fact that if $\epsilon$ is small enough (depending only on $U$), then $|p_{ij}(w)|\leq |p_{0j}(w)|$.

 Therefore using Lemma \ref{lem:pol}, we have
$$
\left|w\in [0,R^{1+\eta}]: |p_{0j}(w)|\leq 10\epsilon\right|\leq C_1(d)R^{(1+\eta-\eta/d)}.
$$
Since $d_G(\exp(wU)g\exp(-wU),e)<10\epsilon$ implies in particular that every coordinate is less than $10\epsilon$ (by taking $\log$ and since Jacobian is close to $1$ around $e$). Let $C_2(d)=C_1(d)$.
This finishes the proof.
\end{proof}

Denote $
L_U=\max\{m_j : j = 1,\dots, n\}
$ to be the depth of the longest chain for $U$.

\begin{lemma}\label{lem:nonsl} For every $\frac{1}{2L_U^2+2L_U+1}>\eta'>0$, every $\epsilon\in [0,\epsilon_0]$ there exist $R_{\epsilon,\eta'}$ such that for every $R\geq R_{\epsilon,\eta'}$, every $g\in \Bow(R,\epsilon,e)
$ such that $a_V(g) = a_X(g) = 0$ and $d_G(g,e) < \ve/C(d)$ and every $t\in [0, R^{1+\eta'}]$ for which
$$
d_G(\exp(tU)g\exp(-tU),e)>10\epsilon,
$$
there exists $\abs{s} \in[0, 2(L_U+1)\eta' \log R]$ and $C_t := a_U(t)U + \sum_{j=0}^n c_j(t)X_0^j$ such that $\max_j \abs{c_j(t)} \ge \ve/n$ and
$$
d_G(\exp(-sX)\exp(tU)g\exp(-tU)\exp(sX),\exp(C_t))<R^{-\eta'},
$$
\end{lemma}

\begin{proof}
Write $g = \exp\left(a_U(g)U + \sum_{j=1}^n \sum_{i=0}^{m_j} a_{ij}(g)X_i^j\right)$, and let
 \[p_{ij}(t,s) = \sum_{k=0}^{m_j-i} \frac{t^k}{k!}e^{-(m_j-2i)s}a_{(k+i)j}\] be the coefficient of $X_i^j$ for $\exp(-sX)\exp(tU)g\exp(-tU)\exp(sX)$ as determined by \eqref{eq:conj-formula}. The condition that $g \in \Bow(R,\ve,e)$ implies that $\abs{a_{kj}} < C(d)\ve k! / R^k$ by Lemma \ref{lem:polynomial-coefficients}. Thus, if $t \le R^{1+\eta'}$,

\[ t^k/k! a_{(k+i)j} \le \frac{R^{(1+\eta')k}}{R^{k+i}}C(d)\ve k! = R^{k\eta' - i}C(d)\ve k! < R^{1/2 - i}\ve/(m_j+1) \]
if $R$ is sufficiently large.

Thus, if $i \ge 1$, $\abs{p_{ij}(t,0)} < R^{-1/2}\ve$. Similarly for sufficiently large $R$, $\abs{p_{0j}(t,0)} \le (m_j+1)C(d)^2\ve R^{m_j\eta'} \le C(d)^2\ve R^{L_U\eta'} < 2R^{(L_U+1)\eta'}\ve/n$ for every $j$

By assumption, $t$ is such that $d_G(\exp(tU)g\exp(-tU),e) > 10 \ve$. Therefore, there exists some $j$ such that $p_{0j}(t,0) \ge 8\ve/n$. Let $D = 2(L_U+1)$, and consider the function $\zeta : s \mapsto (p_{01}(t,s),\dots,p_{0n}(t,s)) \in \R^n$. Since $p_{0j}(t,0) \ge 8\ve/n$ for some $j$, it follows that $\zeta(0) \not\in [0,7\ve/n]^n$. But since $\abs{p_{0j}(t,0)} < 2R^{(L_U + 1)\eta'}\ve/n$, we get

\[\abs{p_{0j}(t,D\eta'\log R)}  = R^{-m_jD\eta'}\abs{p_{0,j}(t,0)}  < 2R^{(L_U+1)(1 - 2m_j)\eta'}\le 2\ve/n\]

for every $j$, provided $m_j \ge 1$. If $m_j = 0$, the only term appearing is $X_0^j$ which is constant in both $t$ and $s$. Therefore its coefficient is bounded by $\ve / C(d)$ by the assumption that $d(z,e) < \ve/C(d)$. Therefore, $\norm{\zeta(D\eta'\log R)} \le 2\ve$.

By continuity of $\zeta$, we may therefore choose $s \in [0,D\eta'\log R]$ such that $\norm{\zeta(s)} = 2\ve$. Then, since $\abs{p_{ij}(t,0)} < R^{-1/2}\ve$, we get that $\abs{p_{ij}(t,s)} \le R^{-(m_j - 2i)D\eta'- 1/2}\ve < R^{-2\eta'}$ for $i \ge 1$ if $\eta'$ is sufficiently small. Setting $C_t = e^{-2s}a_U(z) U + \sum_{j=1}^n p_{0j}(t,s)X_0^j$ gives that:

\begin{multline*}
 d_G(\exp(-sX)\exp(tU)z\exp(-tU)\exp(sX),\exp(C_t)) \le \\\ell \norm{ \sum_{j=1}^n \sum_{i=1}^{m_j} p_{ij}(t,s)X_i^j} \le \ell R^{-2\eta'} \sum m_j< R^{-\eta'}
\end{multline*}

Here, $\ell$ is the Lipshitz constant for $\log : B_G(e,100n\ve) \to \mf g$ and the last inequality holds if $R$ is chosen sufficiently large. Finally, notice that since $\zeta(s) = 2\ve$, there exists some $j$ such that $p_{ij}(t,s) > \ve/n$.
\end{proof}

\section{Proof of Theorem \ref{th1}}
\label{sec:thmref}

In this section we will prove Theorem \ref{th1}. The proof is rather technical and consists of several steps, which we will divide into subsections to improve readability. The following Theorem is a crucial step in the proof of Theorem \ref{th1} since it shows that for most points, being Kakutani close implies that there exists a long block on which they are close in the metric on $G / \Gamma$ (see Definition \ref{def:kakball}).

\textbf{Sequence of Partitions.} Let $(K_m)_{m\in \N}\subset G/\Gamma$ be a family of compact sets such that $\mu((G/\Gamma)\setminus K_m)\to 0$.
 Let $\cP_m$ be a partition of $G/\Gamma$ such that $K_m^c$ is one atom of $\cP_m$ and
the atoms of $\cP_m\cap K_m$ are sets with diameter in $[\frac{1}{2m^2},\frac{1}{m^2}]$ (with smooth boundaries). It is clear that $\vee_{m\geq 1}\cP_m$ generates the $\sigma$-algebra (and we can use Theorem \ref{rat:gen}).
Fix a compact set $K_0$, with $\mu(K_0)\geq 1-10^{-3}$.
Let $\epsilon'>0$ be a small constant fixed from now on, in particular $\epsilon'<\min(\epsilon_1,\epsilon_0)$.
We have the following theorem:
\begin{theorem}\label{prop:longblock}
Let $(\phi_t)$ be the flow generated by $U$ with $GR(U)>3$.
There exists $\delta_{0}>0$ such that for every $0<\delta\leq \delta_{0}$ there exists a set $E=E_\delta\subset G/\Gamma$, $\mu(E_\delta)>99/100$ and $m_\delta, R_\delta\in \N$ such that for every $m\geq m_\delta$, $R\geq R_\delta$ and every $x,y\in E$ which are $(\frac{1}{100},\cP_m)$-matchable there exists $u\in A(x,y)$ such that (see Definition \ref{def:matchsplit})
$$
S(u,\epsilon')\geq R^{1-\delta},
$$
and moreover $x_u=\phi_ux, y_u=\phi_{h(u)}y\in K_0$.
\end{theorem}

The proof of Theorem \ref{th1} can be deduced from Theorem \ref{prop:longblock}. We will first give a conditional proof of Theorem \ref{th1} (assuming that Theorem \ref{prop:longblock} holds) and then prove Theorem \ref{prop:longblock} in a separate section.  The proof of Theorem \ref{th1} is divided into two parts:  (i) the upper bound on the number of balls and (ii) lower bound on the number of balls. It follows that we only need Theorem \ref{prop:longblock} for (ii). In the proof we will use Remark
\ref{compef}.

\textbf{Outline of the proof:} The proof shows first that $e(({\phi_t}),\log) \le GR(U) - 3$ and then that $e((\phi_t),\log) \ge GR(U)-3$. For the proof of the upper bound, we will apply Lemma \ref{lem:B_Gdiv}. To do so, we need to relate $\Kak(R,\ve,y)$ and $B_R(y,\ve,\mc P)$. The key idea is this: $\Kak(R,\ve,y)$ is defined in such a way that all coordinates are controlled in such a way that if $x \in \Kak(R,\ve,y)$, for each $u \in [0,R]$, we can apply $\phi_t$ to bring $\phi_u(x)$ close to $\phi_u(y)$, with $t$ well-controlled. The decay rate of $\Kak$, however, is 1 away from the claimed decay rate for $B_R(y,\ve,\mc P)$. This is because the definition of $\Kak$ assumes that $x$ and $y$ are close initially, while for $B_R(y,\ve,\mc P)$, we only require a forward orbit of $x$ (which is small relative to $R$, but can depend {\it linearly} on $R$) to be close to $y$. This accounts for the slower decay rate of $B_R$ versus $\Kak$.

\textbf{Claim A} is the relationship between $\Kak$ and $B_R$: if $x\in \bigcup_{p=0}^{\epsilon^3R}\phi_{-p}(\Kak(R,\epsilon^5,y))$, then they are in one Kakutani ball, i.e. $x\in B_R(y,\epsilon, \cP_m)$. Then \textbf{Claim B} shows that the forward orbits of $\Kak(R,\ve^5,y)$ do not overlap giving the desired rate with Lemma \ref{lem:B_Gdiv}.

For the lower bound we use Theorem \ref{prop:longblock} to show that if $x\in B_R(y,\epsilon,\cP_m)$, then \eqref{eq:logu} holds. This together with the upper bound in Lemma \ref{lem:B_Gdiv} implies the lower bound.

\begin{proof}[Proof of Theorem \ref{th1}]

We will separately prove the upper bound which will follow from general estimates on asymptotic divergence of orbits and the lower bound which is a consequence of Theorem \ref{prop:longblock}.
Fix $m\in \N$ (this also fixes the partition $\cP_m$) and let $\epsilon>0$, $\epsilon<m^{-3}$. Let $K_{\epsilon^5}\subset F$, with $\mu(K_{\epsilon^5})>1-\epsilon^5$ be as in Lemma \ref{lem:sys} and let  $\kappa=\min(\kappa(\epsilon^5),\epsilon^3)$.

\paragraph{Upper bound on the number of balls.}
Let $V^m_{\epsilon^2}$ be the $\epsilon^2$ neighborhood of the boundary of $\cP_m$. Since the boundaries are smooth, it follows that
$\mu(V^m_{\epsilon^2})={\rm O}(m^2\epsilon^2)$. Applying the ergodic theorem to $(\phi_t)$ and the set $\chi_{V^m_{\epsilon^2}}$, we obtain a set $D_\epsilon$ such that
 $\mu(D_\epsilon)>1-\epsilon^2$ and a number $N_\epsilon>0$ such that for every $R\geq N_\epsilon$ and every $y\in D_\epsilon$, we have
\be\label{eq:bet}
|\{t\in[0,R]\;:\; \phi_t(y)\in V^m_{\epsilon^2}\}|\leq \frac{\epsilon R}{2}.
\ee



The upper bound will follow from the following two claims:

\textbf{Claim A.} For every $y\in D_\epsilon\cap g_{-\frac{1}{2}\log(\epsilon^2 R)}(K_{\epsilon^5})$, every $R\geq N_\epsilon$ and every $x\in G/\Gamma$ if there exists $p\in [0,\kappa^3 R]$ such that if
$$\phi_px\in \Kak(R,\kappa^5,y)$$
then $x\in B_R(y,\epsilon, \cP_m)$ (see Definition \ref{def:Kakutani}).

\textbf{Claim B.} For every $y\in D_\epsilon\cap g_{-\frac{1}{2}\log(\epsilon^2 R)}(K_{\epsilon^5})$ and every $p,q\in [0,\kappa^3 R]$, with $|p-q|\geq 1$, we have
$$
\phi_{-p}(\Kak(R,\kappa^5,y))\cap \phi_{-q}(\Kak(R,\kappa^5,y))=\emptyset.
$$
Before we prove the claims, let us show how they imply the upper bound.

Take $y\in D_\epsilon\cap K_{\epsilon^5}\cap g_{-\frac{1}{2}\log(\epsilon^2 R)}(K_{\epsilon^5})$. By \textbf{Claim A} it follows that
$$
\bigcup_{p\in [0,\kappa^3 R]}\phi_{-p}(\Kak(R,\kappa^5,y))\subset B_R(y,\epsilon, \cP_m).
$$
Therefore by \textbf{Claim B.} and Lemma \ref{lem:B_Gdiv} (since $y\in K_{\epsilon^5}$, we have ${\rm inj}(y)\geq \kappa$ see Lemma \ref{lem:sys}), we have
\begin{multline*}
\mu(B_R(y,\epsilon, \cP_m))\geq \mu\left(\bigcup_{p\in [0,\kappa^3 R]}\phi_{-p}(\Kak(R,\kappa^5,y))\right)\geq \\\kappa^3 R\mu(\Kak(R,\kappa^5,y))\geq
\kappa^{d+3}R^{-GR(U)+3}.
\end{multline*}
Since this holds for every $y\in D_\epsilon\cap K_{\epsilon^5}\cap g_{-\frac{1}{2}\log( \epsilon^2 R)}(K_{\epsilon^5})$ and $\mu(D_\epsilon\cap K_{\epsilon^5}\cap g_{-\frac{1}{2}\log( \epsilon^2 R)}(K_{\epsilon^5}))\geq 1-\epsilon$(since ($g_s$) preserves $\mu$) it follows that for some $C(\epsilon)>0$ depending on $\epsilon$ only and by Remark \ref{compef}, we have
\be\label{eq:Lab}
K_R(5\epsilon,\cP_m)\leq C(\epsilon)R^{GR(U)-3}.
\ee
Therefore $\beta(\log, 5\epsilon,\cP_m)\leq GR(U)-3$ and so by Theorem \ref{rat:gen},
$e((\phi_t),\log)\leq GR(U)-3$.

 Notice moreover, that if $GR(U)=3$, then by \eqref{eq:Lab} it follows that the number of balls does not depend on $R$. Therefore $e((\phi_t),u)=0$ for every function $u\in \mc F$. By Theorem \ref{th:LB} it follows that if $GR(U)=3$, then $(\phi_t)$ is standard.

 So it remains to prove \textbf{Claim A} and \textbf{Claim B.}

\textbf{Proof of Claim A.} Take $y\in D_\epsilon\cap g_{-\frac{1}{2}\log(\epsilon^2 R)}(K_{\epsilon^5})$ and let $p\in [0,\kappa^3 R]$ be such that $\exp(pU)x \in \Kak(R,\kappa^5,y)$.

This by Definition \ref{def:kakball} implies that 
for some $|b|<\kappa^5$, $|a|<\frac{\kappa^5}{R}$, and $g \in \Bow(R,\kappa^5,e)$ satisfying $a_V(g) = a_X(g) = 0$, we have

\be\label{eq:tdec2}
x=\exp(-pU)\exp(aV)\exp(bX)gy.
\ee

Let $\psi(t):=\frac{te^{b}}{e^{-b}-ae^{b}t}$ be as in Lemma \ref{lem:slren}, $h(t)=\psi(t)+p$, $A(x,y):= \{t\in[0,R]\;:\; \phi_t(y)\notin V^m_{\epsilon^2}\}$ (notice that by \eqref{eq:bet}, we have $|A(x,y)|\geq (1-\epsilon)R$).
Moreover (see Remark \ref{rem:inv}), $|h'(t)-1|<\epsilon$ for every $t\in [0,R]$ and hence $h$ satisfies the condition to be a $(\epsilon,\cP_m)$-matching function. We will show that for every $t\in [0,R]$, we have
\be\label{dgf}
d_{G/\Gamma}(\phi_t y,\phi_{h(t)}x)\leq \epsilon^3.
\ee

This by right invariance, the definition of $h(\cdot)$ and \eqref{eq:tdec2} follows by showing
$$
d_G(e,\exp(\psi(t)U)\exp(aV)\exp(bX)\exp(-tU)\exp(tU)g\exp(-tU))\leq\epsilon^3.
$$
Since $g \in \Bow(R,\kappa^5,e)$ it follows that $d_G(\exp(tU)g\exp(-tU),e)<\kappa^5<\epsilon^5$ for $t \in [0,R]$. Moreover,
by  Lemma \ref{lem:slren}, we have
$$
d_G(\exp(\psi(t)U)\exp(aV)\exp(bX)\exp(-tU),e)<\epsilon^4.
$$
The two above inequalities finish the proof of \eqref{dgf}.  By \eqref{dgf}, for every $t\in A(x,y)$ (see \eqref{eq:bet}), we have
$\cP_m(\phi_ty)=\cP_m(\phi_{h(t)}x)$. Since  $|(A(x,y)|\geq (1-\epsilon)R$, it follows that  $x\in B_R(y,\epsilon,\cP_m)$. This finishes the proof of \textbf{Claim A.}

\textbf{Proof of Claim B.} We will argue by contradiction assuming that there exists $x \in \phi_{q-p}(\Kak(R,\kappa^5,y))\cap \Kak(R,\kappa^5,y)$, with $\kappa^3 R\geq |p-q|\geq 1$ and $y\in D_\epsilon\cap g_{-\frac{1}{2}\log(\epsilon^2 R)}(K_{\epsilon^5})$.
This, by the Definition \ref{def:kakball} in particular means (denoting $r=p-q$) that
$$
x=\exp(aV)\exp(bX)gy,
$$
and,
$$
\exp(-rU)x=\exp(a'V)\exp(b'X)g'y,
$$
where $g,g' \in \Bow(R,\kappa^5,e)$ satisfy $a_V(\cdot) = a_X(\cdot) = 0$, $|a|,|a'|\leq \frac{\kappa^5}{R}$ and $|b|,|b'|\leq\kappa^5$. Choose lifts $\tilde{x},\tilde{y} \in G$ of $x,y\in G/ \Gamma$ minimizing $d_G(\tilde{x},\tilde{y})$. In particular since $\kappa < \inj(K_{\epsilon^5})$,

\begin{equation}
\tilde{x} = \exp(aV)\exp(vX)g\tilde{y} \qquad \exp(-rU)\tilde{x}=\exp(a'V)\exp(b'X)g'\tilde{y}\gamma
\end{equation}

for some $\gamma \in \Gamma$. Therefore using the second equality to express $\tilde{y}\gamma$ and the first to express $\tilde{y}^{-1}$, we get

$$
\tilde{y}\gamma \tilde{y}^{-1}= {g'}^{-1}\exp(-b'X)\exp(-a'V)\exp(rU)\exp(aV)\exp(bX)g.
$$

Notice that since $|r|\geq 1$ and all the other terms on the RHS are $\kappa$ small, it follows that $\gamma \neq e$.
Multiplying on the left by $\exp(-sX)$ and on the right by $\exp(sX)$
with $s=\frac{1}{2}\log \epsilon^2 R$ gives 
\begin{multline}\label{eq:mul1}
(\exp(-sX)\tilde{y})\gamma (\exp(-sX)\tilde{y})^{-1}=\exp(-sX){g'}^{-1}\exp(sX)\exp(-b'X)\cdot \\
\exp(e^{2s}a'V)\exp(e^{-2s}rU)\exp(e^{2s}aV)\exp(bX)\exp(-sX)g\exp(sX).
\end{multline}
By the definition of $s$ it follows that $\max(|e^{2s}a'|,|b'|,|e^{-2s}r|,|e^{2s}a|,|b|)\leq \kappa^3$. Moreover, by Lemma \ref{lem:dirs}  (with $\delta'=0$) it follows that for $w\in\{ g,g'^{-1}\}$ (since each such $w\in \Bow(R,\kappa^5,e)$)
$$
d_G(\exp(-sX)w\exp(sX),e)<\kappa^{4}.
$$

Therefore, the RHS of \eqref{eq:mul1} is $\kappa^2$ close to $e$. However by definition, $g_sy\in K_{\epsilon^5}$ and hence, by Lemma \ref{lem:sys}, it follows that
$d_G((\exp(-sX)\tilde{y})\gamma (\exp(-sX\tilde{y}))^{-1},e)\geq \kappa$. This contradiction finishes the proof of the upper bound.

\paragraph{Lower bound on the number of balls.}  Notice that from the upper bound estimates (in particular, \eqref{eq:Lab}), it follows that if $GR(U)=3$, then $(\phi_t)$ is standard. Hence in what follows we assume that $GR(U)>3$, which, by Lemma \ref{th2} is equivalent to $\dim G-\dim(C(X))-3>0$ and we can use Theorem \ref{prop:longblock}.

Fix $\delta>0$, $0 < \ve < 1/100$, and $m \ge m_\delta$ satisfying $1/m^2 < \ve'/C(d)$ and $R\geq R_\delta$ and assume that $x,y\in E_\delta$ are $(\epsilon,\cP_m)$-matchable.
Using Theorem \ref{prop:longblock} and Definition \ref{def:matchsplit} it follows that there exists $p,q\in [0,R]$ (in fact $p=u$ and $q=h(u)$ where $h$ is the matching function) such that
$$
\phi_px\in \Kak(R^{1-\delta}, \epsilon',\phi_q y).
$$
This implies that

$$
x\in \phi_{-p}\left(\Kak(R^{1-\delta}, \epsilon',\phi_q y)\right).
$$

Therefore,
$$
B_R(y,\epsilon,\cP_m)\cap E_\delta \subset \bigcup_{p,q}\phi_{-p}\left(\Kak(R^{1-\delta}, \epsilon',\phi_q y)\right).
$$
Since $\phi_qy\in K_0$ (by Theorem \ref{prop:longblock}), we have ${\rm inj}(\phi_qy)>c_0=\inj(K_0)$.
Therefore, since $\epsilon'$ is small enough (in particular $\epsilon'<c_0$), it follows by Lemma \ref{lem:B_Gdiv} that
$$
\mu(B_R(y,1/100,\cP_m)\cap E_\delta)\leq R^2 \max_{q \in [0,R]}\mu(\Kak(R^{1-\delta}, \epsilon',\phi_q y)\leq c(\epsilon')R^{2-(1-\delta)(GR(U)-2)}.
$$
So the number of balls needed to cover $1-\epsilon$ of space is at least $R^{(1-\delta)(GR(U)-2)-2}$.  Therefore $\beta(\log,\epsilon,\cP_m)\geq (1-\delta)(GR(U)-2)-2$ and since the sequence $(\cP_m)$ is generating $e((\phi_t),\log)\geq (1-\delta)(GR(U)-2)-2$. The proof for general $\Gamma$ is finished by taking limit as $\delta$ goes to $0$.

Now assume that $\Gamma$ is cocompact and let $(R_i)_{i=1}^{R^{2\delta}}$ be given by $R_i=iR^{1-2\delta}$. We will show that any $x,y\in E_\delta$ which are $(1/100,\cP_m)$-matchable have to satisfy
 \be
 \label{eq:logu}
x\in \bigcup_{p\in [-2R,2R]}\phi_{-p}\left(\bigcup_{i=1}^{R^{2\delta}}
\Kak(R^{1-4\delta}, \epsilon'^{1/3},\phi_{R_i}y)\right)
\ee
Before we give the proof of \eqref{eq:logu}, let us show how it implies the lower bound. By \eqref{eq:logu}, we have
$$
B_R(y,\epsilon,\cP_m)\cap E_\delta \subset \bigcup_{p\in [-2R,2R]}\phi_{-p}\left(\bigcup_{i=1}^{R^{2\delta}}
\Kak(R^{1-4\delta}, \epsilon'^{1/3},\phi_{R_i}y)\right).
$$
Since $\ve' < \inj(G/\Gamma)$, by Lemma \ref{lem:B_Gdiv}, for some $c(\epsilon')>0$ (depending on $\ve'$ only)

\begin{multline*}
\mu(B_R(y,\epsilon,\cP_m)\cap E_\delta)\leq 4R^{1+2\delta}\max_{i}(\mu(\Kak(R^{1-4\delta}, \epsilon'^{1/3},\phi_{R_i}y)))\leq \\
c(\epsilon')R^{1+2\delta}R^{-(1-4\delta)(GR(U)-2)}.
\end{multline*}
This gives
$$
\mu(B_R(y,\epsilon,\cP_m)\cap E_\delta)\leq c(\epsilon')R^{U(\delta)},
$$
where $U(\delta)=1+2\delta-(1-4\delta)(GR(U)-2)$. Hence the number of balls needed to cover $1-\epsilon$ of space, i.e.  $K_R(\epsilon,\cP_m)$ is at least $C(m)R^{-U(\delta)}$ (for some constant depending on $m$ only). Therefore $\beta(\log,\epsilon,\cP_m)\geq -U(\delta)$ and since the sequence $(\cP_m)$ is generating $e((\phi_t),\log)\geq -U(\delta)$. The proof is finished by taking limit as $\delta$ goes to $0$, since $-U(0)=GR(U)-3$. So it remains to show \eqref{eq:logu}.

Using Theorem \ref{prop:longblock} and Definition \ref{def:matchsplit} it follows that there exists $p,q\in [0,R]$ (in fact $p=u$ and $q=h(u)$ where $h$ is the matching function) such that
$$
\phi_px\in \Kak(R^{1-\delta}, \epsilon',\phi_q y)
$$
Let $R_j$ be the number minimizing $|q-R_i|$ (over all $i$). To finish the proof of \eqref{eq:logu} it is enough to show that there exists $\ell=\ell(p,q)$, $|\ell|\leq R$ such that (since $|p+\ell|<2R$)
$$
\phi_{p+\ell}x\subset \Kak(R^{1-2\delta},\epsilon'^{1/3}, \phi_{R_j}(y)).
$$
This however follows by Lemma \ref{lem:adu} with $\epsilon^3=\epsilon'$, 
$z=\phi_px$, $y=\phi_qy$, $L=R_j-q$. This finishes the proof of Theorem \ref{th1}.

\end{proof}

So it remains to prove Theorem \ref{prop:longblock}

\section{Proof of Theorem \ref{prop:longblock}}
\label{sec:technical}
\textbf{Outline of the proof:} Assume that $x,y\in G\slash \Gamma$ are $(1/100,\cP_m)$-matchable. Then every matching arrow (of $\phi_ux$ and $\phi_{h(u)}y$) can be parametrized by $j=j(u)\in \N$ and $w\in \{1,2\}$: $j\in \N$  measures the splitting time of $\phi_ux$ and $\phi_{h(u)}y$ in exponential scale (see the set $C_{j,R,m}(x,y)$ below). Moreover $w\in \{1,2\}$ gives the direction which is responsible for the splitting, i.e. if $w=1$ then the splitting is definitely produced by directions different than $V$ and if $w=2$ then the splitting might be (but not necesarilly has to be) produced by $V$ (see \eqref{eq:aj1} and \eqref{eq:aj2}). Propositions \ref{prop:splitlong} and \ref{lem:splt} are purely of combinatorial nature (no dynamics involved, just a counting argument). Proposition \ref{prop:splitlong} states, that if for every $j$ (sufficiently large) the measure of arrows with label $j,w$, for $w\in \{1,2\}$ is exponentially small (see \textbf{b.}) than the total measure of the matching is also small (since the series is summable over $j$). Proposition \ref{lem:splt} states that if in every window of size $2^{j(1+c\delta)}$ for $c=2$ (in (A)) and $c=40$ (in (B)) the relative measure of arrows with label $j$ is exponentially small, then the total measure of arrows with label $j$ has to be small.
\\

Fix $\delta>0$, $m>0$ and $R,j\in \R_+$. Assume that $x,y\in G/\Gamma$ are $(1/100,\cP_m)$-matchable. Let $A(x,y)\subset [0,R]$ denote the matching set and $h:A(x,y)\to [0,R]$ the matching function. We define two sets which will play a crucial role in the proof. Recalling Definitions and \ref{def:matchsplit} and \ref{def:splittime}, we define
$$
C_{j,R,m}(x,y):=\{u\in A(x,y)\;:\;d_G(x_u,y_u)<2m^{-2}\text{ and } 2^{j}\leq S(u,\epsilon')<2^{j+1}\}.
$$
Let (see Definition \ref{def:matchsplit} and \eqref{emj}, \eqref{emj2})
\be\label{eq:aj1}
C_{j,R,m}^1(x,y):=C_{j,R,m}\cap\{u\in A(x,y)\;:\;x_u\in \Kak^{1,\delta}(2^j,\epsilon',y_u)\}
\ee
and
\be\label{eq:aj2}
C_{j,R,m}^2(x,y):=C_{j,R,m}\cap\{u\in A(x,y)\;:\;x_u\in \Kak^{2,\delta}(2^j,\epsilon',y_u)\}.
\ee
By definition, $C_{j,R,m}(x,y)=C_{j,R,m}^1(x,y)\cup C_{j,R,m}^2(x,y)$.
\begin{remark}\label{rem:nwa} Recall that the partition $(\cP_m)$ is given by a compact set $K_m\subset G/\Gamma$ (and we divide $K_m$ into sets of diameter $\in [\frac{1}{2m^2}, \frac{1}{m^2}]$). It follows that for $u\in A(x,y,)$ satisfying $x_u=\phi_ux\in K_m$ and $y_u=\phi_{h(u)}y\in K_m$, we have $d_{G/\Gamma}(x_u,y_u)\leq \frac{c}{m^2}$ and hence $x_u\in \operatorname{Kak}(R_m,\epsilon',y_u)$, where the $R_m$ grows to $+\infty$ with $m$. Therefore, relatively on the compact set $K_m$, the sets $\{C_{j,R,m}(x,y)\}_{j\geq R_m}$ partition the matching.
\end{remark}
Let $d$ and be as in Lemma \ref{lem:pol}.  The following proposition implies Theorem \ref{prop:longblock}:
\begin{proposition}\label{prop:splitlong}There exists $c(d)>0$ and $\delta_0>0$ such that for every $\delta_0>\delta>0$ there exists a set $E_\delta\subset G/\Gamma$, $\mu(E_\delta)\geq \frac{99}{100}$ and $m_\delta, R_\delta\in \N$ such that for every $m\geq m_\delta$, $R\geq R_\delta$ and every $x,y\in E_\delta$ there exist $W_R(x,y)\subset A(x,y)$ satisfying:
\begin{enumerate}
\item[(\textbf{a}).] $|W_R(x,y)|\geq \frac{99}{100}R$;
\item[(\textbf{b}).] for every $p\in W_R(x,y)$, we have $\phi_px, \phi_{h(p)}y\in K_0$  (recall that $K_0\subset K_m$ is a fixed compact set and $K_m$ is the compact part of $\cP_m$);
\item[(\textbf{c}).]for every $j\in \N$ satisfying $2^j\leq R^{1-40\delta}$ any $(1/100,\cP_m)$-matching of $x$ and $y$ (with matching function $h$), we have for $w=1,2$
$$
\left|C_{j,R,m}^w(x,y)\cap W_R(x,y)\cap h^{-1}(W_R(x,y))\right|\leq \frac{c(d)R}{2^{d^{-1}\delta j}}.
$$
\end{enumerate}
\end{proposition}

We will prove Proposition \ref{prop:splitlong} in a separate subsection. Let us now show how the above proposition implies Theorem \ref{prop:longblock}.
\begin{proof}[Proof of Theorem \ref{prop:longblock}]

Fix $\delta>0$, $R\geq R_{\delta}$, $m\geq m_{\delta}$, $x,y\in E_\delta$ and a  $(\frac{1}{100},\cP_m)$-good matching of $x$ and $y$ with the set $A(x,y)$ and the matching function $h$. Notice that by \textbf{(b)}, the definition of $K_m$ (the atoms of $\cP_m\cap K_m$ have diameter less that $m^{-2}$) and the definition of $C_{j,R,m}(x,y)$, we have
$$
A(x,y)\cap W_R(x,y)\subset \bigcup_{j\in \N} C_{j,R,m}(x,y).
$$
Moreover, by \eqref{lem:contphi}
and the definition of $C_{j,R,m}(x,y)$, for $j\leq \log_2 f(m)$ (See Remark \ref{rem:nwa}),
\begin{equation}\label{wpan}
C_{j,R,m}(x,y)=\emptyset.
\end{equation}


Hence, by (\textbf{a}), (\textbf{b}) and the definition of $C_{j,R,m}(x,y)$, we have (recall also that $h$ is the matching function, hence it is absolutely continuous)
\begin{equation}\label{ala}
\begin{aligned}
\frac{1}{m}&>\bar{f}_R^{P_m}(x,y)\\
&\geq 1-\frac{|((W_R(x,y))^c\cup h^{-1}(W_R(x,y))^c)\cap[0,R]|}{R}-\\
&\frac{1}{R}\sum_{j\geq0}|C_{j,R,m}(x,y)\cap W_R(x,y)\cap h^{-1}(W_R(x,y))|\\
&\geq\frac{9}{10}-\frac{1}{R}\sum_{j\geq0}|C_{j,R,m}(x,y)\cap  W_R(x,y)\cap h^{-1}(W_R(x,y)) |
\end{aligned}
\end{equation}

Let $j_R$ be such that,
\begin{equation}\label{er}
2^{j_R}\leq R^{1-40\delta}<2^{j_R+1}.
\end{equation}

By  \eqref{wpan}, (\textbf{c}) and $C_{j,R,m}(x,y)=C_{j,R,m}^1(x,y)\cup C_{j,R,m}^2(x,y)$, we have
\begin{equation}
\begin{aligned}
\frac{1}{R}\sum_{j< j_R}|C_{j,R,m}(x,y)\cap W_R(x,y)\cap h^{-1}(W_R(x,y))|&\leq\\
\frac{1}{R}\sum_{\log f(m)\leq j\leq j_R-1}\frac{2c(d)R}{2^{d^{-1}\delta j}}
&\leq\frac{1}{1000},
\end{aligned}
\end{equation}
by enlarging $m$ if necessary (since $f(m)$ goes to $\infty$).
Therefore and by \eqref{ala} there exists $j_1\geq j_R$ such that
\begin{equation}
C_{j_1,R,m}(x,y)\cap W_R(x,y)\cap h^{-1}(W_R(x,y))\neq\emptyset.
\end{equation}
By definition of $C_{j_1,R,m}(x,y)$ and \eqref{er} it follows that there exists $u\in A(x,y)$ such that
\begin{equation}
S(u,\epsilon')\geq 2^{j_R}\geq\frac{1}{2}R^{1-40\delta}\geq R^{1-41\delta}.
\end{equation}
Since $\delta>0$ is arbitrary (changing $\delta'=41\delta$), this finishes the proof of Theorem \ref{prop:longblock}.
\end{proof}

\subsection{Proof of Proposition \ref{prop:splitlong}}
We will formulate a proposition which will imply Proposition \ref{prop:splitlong}.
\begin{proposition}\label{lem:splt} There exists $c'(d)>0$ and $\delta_0>0$ such that for every $\delta_0>\delta>0$  there exists a set $E_\delta\subset G/\Gamma$, $\mu(E_\delta)\geq \frac{99}{100}$ and $m_\delta, R_\delta\in \N$ such that for every $m\geq m_\delta$, $R\geq R_\delta$ and every $x,y\in E_\delta$ there exist $W_R(x,y)\subset [0,R]$ such that (\textbf{a}) and (\textbf{b}) holds for every $(1/100,\cP_m)$-matching of $x$, $y$ (with matching function $h$) and for every $j\in \N$ satisfying $2^j\leq R^{1-40\delta}$, we have
\begin{enumerate}
\item[(A)] for every $u\in W_R(x,y)$ (see Definition \ref{def:balls})
$$
|C_{j,R,m}^1(x,y)\cap B(u,2^{(1+2\delta)j})\cap W_R(x,y)\cap h^{-1}(W_R(x,y))|\leq c'(d) 2^{(1+2\delta-2\delta/d)j};
$$
\item[(B)] for every $u\in W_R(x,y)$
$$
|C_{j,R,m}^2(x,y)\cap B(u,2^{(1+40\delta)j})\cap W_R(x,y)\cap h^{-1}(W_R(x,y))|\leq 2^{(1+20\delta)j};
$$

\end{enumerate}

\end{proposition}

Before we prove Proposition \ref{lem:splt} let us show how it implies Proposition \ref{prop:splitlong}.

\begin{proof}[Proof of Proposition \ref{prop:splitlong}]
Notice that by assumptions of Proposition \ref{lem:splt} it follows that we only need to prove (\textbf{c}) in Proposition \ref{prop:splitlong} (with $w=1,2$ in (\textbf{c})). The proof for $w=1$ uses (A) and the proof for $w=2$ uses (B). Since the proofs in both cases follow the same lines, we will give the proof in case $w=1$.
Fix $j$ as in \textbf{(c)}. Divide the interval $[0,R]$ into disjoint intervals $I_1,\ldots,I_k$ of length $2^{(1+2\delta)j}$ in the following way. Fix the smallest element $$u_1\in C_{j,R,m}^1(x,y)\cap W_R(x,y)\cap h^{-1}(W_R(x,y)).$$
 Let $I_1$ be an interval with right endpoint $u_1$ and length $l_j^1:=2^{(1+2\delta)j}$ . Now inductively for $u> 1$, we pick $u_w$ to be the smallest element in $C_{j,R,m}^1(x,y)\cap W_R(x,y)\cap h^{-1}(W_R(x,y))\setminus(I_1\cup...\cup I_{u-1})$. As $u_w$ satisfies $(A)$, we let $I_w$  be the interval with right endpoint $u_w$ and length $l_j^1$. We continue until we cover $C_{j,R,m}^1(x,y)\cap W_R(x,y)\cap h^{-1}(W_R(x,y))$.

Since $2^j\leq R^{1-40\delta}$, we have $2^{(1+2\delta) j}<R^{1-\delta}\ll R$ and hence $k>1$. Moreover by definition, we have
\begin{equation}\label{qwe}
k \leq \left[\frac{R}{l_j^1}\right] +2.
\end{equation}

Notice that by Definition \ref{def:balls}, the fact that $C_{j,R,m}^1(x,y)\subset A(x,y)$ and the definition of $(I_i)$ it follows that
\begin{multline*}
C_{j,R,m}^1(x,y)\cap I_i\cap W_R(x,y)\cap h^{-1}(W_R(x,y))\subset \\
\subset C_{j,R,m}^1(x,y)\cap B(u_i,2^{(1+2\delta)j})\cap W_R(x,y)\cap h^{-1}(W_R(x,y))
\end{multline*}
Therefore and by (A), we have
\begin{equation}\label{eq1}
|C_{j,R,m}^1(x,y)\cap I_i\cap W_R(x,y)\cap h^{-1}(W_R(x,y))|\leq c'(d)2^{(1+2\delta-2\delta/d)j}.
\end{equation}

Summing over $i\in\{1,\ldots,k\}$ by \eqref{eq1} and \eqref{qwe},  we get
$$
| C_{j,R,m}^1(x,y)\cap W_R(x,y)\cap h^{-1}(W_R(x,y)|\leq
(\frac{R}{l_j^1}+2)c'(d)2^{(1+2\delta-2\delta/d)j}\leq\frac{c(d)R}{2^{d^{-1}\delta j}},
$$
the last inequality by the definition of $l_j^1$ and defining $c(d)=2c'(d)$.
This finishes the proof.

\end{proof}

\subsection{Proof of Proposition \ref{lem:splt}}
In this section we will prove Proposition \ref{lem:splt}. It is the most technical part of the paper. We will start by giving an outline of the proof:

\textbf{Outline of the proof of Proposition \ref{lem:splt}:} The arguments are very different in proving (A) and (B). All the difficulty in proving (B) is transferred to Lemma \ref{lem:mesem}. Indeed,  Lemma \ref{lem:mesem}  directly implies Corollary \ref{cor:mesem} which implies (B). We will give the outline of proof of Lemma \ref{lem:mesem} in the next section. The method in proving (A) is based on the fact that if $x\in \Kak^{1,\delta}(2^j,\epsilon',y)$ then the first non-orbit divergence occurs in a direction other than $V$, $X$ or $U$ (the $V$ coordinate is too small by definition of $\Kak^{1,\delta}$) and consequently the direction in which $x,y$ split belongs to the centralizer $C(U)$ but is different than the flow direction. Since $\bar{f}$- metric allows one to ``slide'' along the orbits of points (ie, $U$ direction) only, there is no way to correct the splitting in the $C(U)$ direction if it is different than the flow (which is the case for $(A)$). Let us also stress out that the condition $\dim G-\dim(C(X))-3>0$ is only used in the proof of Lemma \ref{lem:mesem}.

We divide the proof in several subsections.
For $j\in \N$ let (see \eqref{emj})
\be\label{eq:jbadset}
G(\delta,j,y)=\Kak^{2,\delta}(2^j,\epsilon',y)\cap\left(\bigcup_{p,q\in[2^{j(1+20\delta)},2^{j(1+40\delta)+1}]} \phi_{-p}(\Kak^{2,\delta}(2^j,\epsilon',\phi_qy))\right),
\ee

We have the following lemma:
\begin{lemma}\label{lem:mesem} For $j\in \N$ let
$$
EG(\delta,j)=\{y\in G / \Gamma \;:\; G(\delta,j,y)= \emptyset\}.
$$
There exists $j_\delta>0$ such that for $j\geq j_\delta$,
 $$
\mu(EG(\delta,j))\geq 1-j^{-2}.
$$
\end{lemma}

By the above lemma it follows that if we define
\be\label{def:e1}
E_\delta':=\bigcap_{j>j'_\delta}EG(\delta,j),
\ee
then $\mu(E_\delta')\geq 1-10^{6}$ if $j'_\delta:=\max(j_\delta, 10^{10})$.

We will prove Lemma \ref{lem:mesem} in the last section, let us first state the following immediate corollary:
\begin{corollary}\label{cor:mesem} For $y\in E_\delta'$ and
$x\in \Kak^{2,\delta}(2^j,\epsilon', y)$ with $j>j_\delta$, we have
$$
\phi_px \notin \Kak^{2,\delta}(2^j,\epsilon',\phi_qy),
$$
for any $p,q\in [2^{j(1+20\delta)},2^{j(1+40\delta)+1}]$.
\end{corollary}
\begin{proof}
This just follows by the definition of $E'_\delta$, since for $y\in E'_\delta$, we have $G(\delta, j,y)=\emptyset$ for $j\geq j_\delta$.
\end{proof}

We can now prove Proposition \ref{lem:splt}:
\begin{proof}[Proof of Proposition \ref{lem:splt}]
Let $K_0$ be a fixed compact set of positive measure (see \eqref{lem:sys}). By ergodic theorem for $\chi_{K_0}$ there exists a set $ER_\delta$ and $N_\delta^1$ such that for every $y\in ER_\delta$ and every $|N|\geq N_\delta^1$, we have
\be\label{eq:esa}
\{g_sy\;:\; s\in [N,(1+100\delta)N]\}\cap K_0\neq \emptyset.
\ee

Let $E'_\delta=E^1_\delta\cap K_0\cap ER_\delta$. By ergodic theorem for $\chi_{E'_\delta}$ it follows that there exists a set $E_\delta \subset E_\delta' \subset F$, $\mu(E_\delta)>999/1000$ and $N_\delta>0$ such that for every $x,y\in E_\delta$ and every $R\geq N_\delta$, we have

$$
|W_R(x,y):=W_R(x)\cap W_R(y)|\geq 99/100,
$$
where, for $z\in E_\delta$,
$$
W_R(z):=\{t\in [0,R]:\phi_tz\in E'_\delta\cap K_0\}.
$$

Notice also that if $u\in h^{-1}(W_R(x,y))$, then for every $|N| \geq \max(N_\delta,N_\delta^1)$, we have  $y_u=\phi_{h(u)}y\in E'_\delta$ and, by \eqref{eq:esa}, that
\be\label{comps}
\{g_sy_u\;:\; s\in [N,(1+100\delta)N]\}\cap K_0\neq \emptyset,
\ee
If $s_0$ is in the intersection, then by Lemma \ref{lem:sys}, it follows that
\be\label{sys2}
\inf_{\gamma\in \Gamma}d_G(g_{s_0}y_u\gamma (g_{s_0}y_u)^{-1},e)>c_0.
\ee

Notice that by the definition of $W_R(x,y)$ and the set $E'_\delta$ it follows that (\textbf{a}) and (\textbf{b}) hold.  Hence we only need to show that (A) and (B) in Proposition \ref{lem:splt} hold. The methods of proof are different for (A) and (B), (B) being a simple consequence of Corollary \ref{cor:mesem}.

\textbf{Proof of (B):} Notice that by the definition of $C^2_{j,R,m}(x,y)$ (see also Definition \ref{def:matchsplit}) it follows that if $u\in C^2_{j,R,m}\cap W_R(x,y)\cap h^{-1}(W_R(x,y))$, then $y_u=\phi_{h(u)} y\in E'_\delta\subset E_\delta$ and
$$
x_u\in \Kak^{2,\delta}(2^j,\epsilon',y_u).
$$
Notice that since $h$ is an $(1/100,\cP_m)$-matching function it  follows that for every $v-u\in [2^{j(1+20\delta)},2^{j(1+40\delta)}]$ and $v\in A(x,y)$, we have
$$
|h(v)-h(u)|<2|v-u|.
$$
Therefore and by Corollary \ref{cor:mesem}  for $y=y_u$ and $x=x_u$ it follows that for any $v\in [u+2^{j(1+20\delta)},u+2^{j(1+40\delta)}]$ (notice that $|h(v)-h(u)|<2^{j(1+40\delta)+1})$
$$
\phi_{v-u}(x_u)\notin \Kak^{2,\delta}(2^j,\epsilon',\phi_{h(v)-h(u)}(y_u)).
$$
Note that $\phi_{v-u}(x_u)=x_{v}$ and $\phi_{h(v)-h(u)}(y_u)=\phi_{h(v)}y=y_{v}$.

Therefore
$$
 C^2_{j,R,m}(x,y)\cap \left(B(u,2^{j(1+40\delta)})\setminus B(u,2^{j(1+20\delta)})\right)=\emptyset.
$$
So
$$
|C^2_{j,R,m}(x,y)\cap B(u,2^{j(1+40\delta)})\cap W_R(x,y)\cap h^{-1}(W_R(x,y))|\leq 2^{j(1+20\delta)}.
$$
and this finishes the proof of (B).

\textbf{Proof of (A):} Let $u$ be the smallest element in
$C^1_{j,R,m}(x,y)\cap B(u,2^{j(1+2\delta)})\cap W_R(x,y)\cap h^{-1}(W_R(x,y))$ (if such $u$ doesn't exist, then the intersection is empty and the proof is finished).

We will show that there exists a set $ A'_u\in [h(u),h(u)+2^{j(1+2\delta)}]$, $|A'_u|\leq C'(d)2^{j(1+2\delta-2\delta/d)} $, such that for every $h(v)\in [h(u),h(u)+2^{j(1+2\delta)}]\setminus A'_u$,
\be\label{vnotin}
v\notin C^1_{j,R,m}(x,y)\cap B(u,2^{j(1+2)\delta})\cap W_R(x,y)\cap h^{-1}(W_R(x,y)).
\ee
This will finish the proof of (A) since $|h'-1|<\epsilon$ on $A(x,y)$.

Assume for contradiction that \eqref{vnotin} does not hold and let $v$ belong to the RHS of \eqref{vnotin}. By the definition of $C^1_{j,R,m}(x,y)$ and Definition \ref{def:matchsplit} and \ref{def:splittime} it follows that for $w=u,v$, we have
\be\label{eq:newqe}
x_w=\exp(a_wV)\exp(b_wX)g_w y_w,
\ee
where $|a_w|<2^{-j(1+10\delta)}$, $|b_w|<m^{-2}$ and $g_w\in \Bow(2^j,\epsilon',e)$ and $a_V(g_w) = a_X(g_w)= 0$  with $d_G(g_w,e)<m^{-2}$. Moreover, since $|a_w|\leq 2^{-j(1+10\delta)}$, by Definition \ref{def:matchsplit}, Definition \ref{def:splittime} and Definition \ref{def:kakball}, it follows that\footnote{The splitting time is defined through the $V$ coefficient and dynamical control on $g_w$. Since we are in the set $C^1$ (see the definition of $\Kak^{1,\delta}$) it follows that the $V$ coefficient is of lower order, hence the splitting has to be produced by $g_w$.}
\be\label{spzw}
g_w\notin \Bow(2^{j+1},\epsilon',e)
\ee
Notice that $\ve'$ is fixed, so we may choose $m$ sufficiently large so that the condition \\ $d_G(g_w,e) < m^{-2}$ will imply that $d_G(g_w,e) < C(d)^{-1}\ve'$, where $C(d)$ is as in Corollary \ref{corpol}. Applying this Corollary with $R=2^{j+1}$ and $\epsilon=\epsilon'$ it follows that there exists $A_u\subset [0,2^{j(1+2\delta)}]$, $|A_u|\leq C(d)2^{j(1+2\delta-2\delta/d)}$ and such that for every $r\in [0,2^{j(1+2\delta)}]\setminus A_u$, we have
$$
d_G(\exp(rU)g_u\exp(-rU),e)>10\epsilon'.
$$
Hence, by Lemma \ref{lem:nonsl}, with $R=2^j$, and $\eta'=2\delta$ we have that for every $r\in [0,2^{j(1+2\delta)}]\setminus A_u$ there exists $s_r \in [0,4(L_U+1)j\delta]$,
\be\label{crueq}
d_G(\exp(-s_rX)\exp(rU)g_u\exp(-rU)\exp(s_rX),\exp(C_r))<2^{-\delta j},
\ee
 where $C_r \in C(U)$ is as in the statement of Lemma \ref{lem:nonsl}.

Notice that $x_v=\exp((v-u)U)x_u$ and $y_v=\exp((h(v)-h(u))U)y_u$. Therefore and by \eqref{eq:newqe} (setting $p=v-u$ and $q=h(v)-h(u)$) it follows that if $\tilde{x}$ and $\tilde{y}$ are lifts of $x$ and $y$, for some $\gamma\in \Gamma$,
\be\label{impdec}
\begin{aligned}
\exp(pU)\exp(a_uV)&\exp(b_uX)g_u\exp(h(u)U)\tilde{y}\\&=\exp(a_vV)\exp(b_vX)g_v\exp(qU)\exp(h(u)U)\tilde{y}\gamma
\end{aligned}
\ee

We will consider two cases:

\textbf{Case I.}  $\gamma=e$ in \eqref{impdec}. In this case we get
\be\label{osteq}
\exp(pU)\exp(a_uV)\exp(b_uX)g_u\exp(-qU)=\exp(a_vV)\exp(b_vX)g_v
\ee
By Lemma \ref{lem:slren} and since $|a_u|<2^{-j(1+10\delta)}$, it follows that there exists $l=l(p)$ with $p/2<|l|<2p$, such that
$$
\exp(pU)\exp(a_uV)\exp(b_uX)=\exp(b_pV)\exp(c_pX)\exp(lU),
$$
and $|b_p|<2|a_u|$, $|c_p|<4m^{-2}$.
Therefore, using \eqref{osteq}, we have
\be\label{lqU}
\exp((l-q)U)\exp(qU)g_u\exp(-qU)=\exp(-c_pX)\exp(-b_pV)\exp(a_vV)\exp(b_vX)g_v.
\ee
If $q\in [0,2^{j(1+2\delta)}]\setminus A_u$, then by \eqref{crueq}, we have, that there exists $s\in [0,4(L_U+1)j\delta]$ such that
\be\label{eq:jkl}
d_G(\exp(-sX)\exp(qU)g_u\exp(-qU)\exp(sX),\exp(C))<2^{-j\delta}.
\ee

Conjugate \eqref{lqU} by $\exp(sX)$. Since $s\leq 4(L_U+1)j\delta \leq \frac{1}{2}\log 2^j$ (since $\delta$ is of lower order than $L_U$), by Lemma \ref{lem:B_Ggr}, we have
$$
d_G(\exp(-sX)g_v\exp(sX),e)<m^{-1},
$$
since $\|(z_v)_C\|\leq d_G(z_v,e)\leq m^{-2}$.
 Moreover  since $|a_v|,|b_p|\leq 2^{-(1+10\delta)j}$and $|b_v|,|c_p|\leq 4m^{-2}$ and since $s\leq 4(L_U+1)j\delta$, the RHS of \eqref{lqU} after conjugating by $\exp(sX)$ (term by term) is $40m^{-2}$ close to $e$.
 Hence and by \eqref{eq:jkl}, we get
 $$
d_G(\exp(e^{-2s}(l-q)U),\exp(-C))\leq  40m^{-2},
$$
 this however contradicts the properties of $C$ (see Lemma \ref{lem:nonsl}), if $m$ is chosen so that $\epsilon'>m^{-1/100}$. Consequently, for $q\in [0,2^{j(1+2\delta)}]\setminus A_u$ it follows that \eqref{eq:jkl} does not hold. Recall that $q=h(v)-h(u)$. This shows that for  $v\in h^{-1}([0,2^{j(1+2\delta)}]\setminus A_u)$,  \eqref{vnotin} holds. It remains to define $A'_u=h^{-1}(A_u)$ and notice that $|A'_u|\leq (1+\epsilon)|A_u|$ (since $h$ is a matching function).

\textbf{Case II.} $\gamma\neq e$ in \eqref{impdec}. In this case \eqref{impdec} is equivalent to
\begin{multline}\label{eq:neqe}
\exp(-qU)(g_v)^{-1}\exp(-b_vX)\exp(-a_vV)\exp(pU)\exp(a_uV)\exp(b_uX)g_u= \\ \exp(h(u)U)\tilde{y}\gamma \tilde{y}\exp(-h(u)U)
\end{multline}
Notice that for $s\in [\frac{1}{2} \log2^{j(1+3\delta)},\frac{1}{2}\log 2^{j(1+9\delta)}]$, we have
\be\label{eq:neq2}
d_G(\exp(-sX)\exp(pU)\exp(a_uV)\exp(b_uX)g_u\exp(sX),e)<\frac{8}{m^{1/3}}.
\ee
Indeed, this follows from applying the conjugation term-wise and by  $|p|<2^{j(1+2\delta)}$, $|a_u|<2^{-j(1+10\delta)}$, $|b_u|\leq m^{-2}$ and $d_G(\exp(-sX)g_u\exp(sX),e)<\frac{1}{m^{1/3}}$ (see Lemma \ref{lem:dirs} with $\delta'=3\delta$). By an analogous reasoning, \eqref{eq:neq2} holds also for $v$ (and $g_{v}^{-1}$ and $-q$ instead of $g_u$ and $p$). So by \eqref{eq:neqe}, for every $s\in [\frac{1}{2} \log2^{j(1+3\delta)},\frac{1}{2}\log 2^{j(1+9\delta)}]$, we have
$$
d_G((g_{-s}y_u)\gamma'(g_{-s}y_u)^{-1},e)<\frac{16}{m^{1/3}},
$$
and this contradicts \eqref{comps} (since $y_u\in ER_\delta$)  with $N= \frac{1}{2} \log2^{j(1+3\delta)}$ if $m$ and $j$ are large enough (see also \eqref{sys2}).
\end{proof}

\subsection{Proof of Lemma \ref{lem:mesem}}
Let us start by giving  outline of the proof:

\textbf{Outline of the proof:} We start by \eqref{xbel}, which is the definition of $G(j,y)$ not being empty. Next, we transform \eqref{xbel} (using the bounds on coefficients) to \eqref{eq:mesga}. Then using the properties \textbf{P1} and \textbf{P2} (which we prove in the appendix) we further transform it to \eqref{eq:mesga3}. We then have Lemma \ref{lem:mesem2}, which tells us that the set of solutions to \eqref{eq:mesga3} is $\frac{1}{T(j)j^{3}}$ small. Then using theorems on the cardinality of lattice points in balls in semisimple Lie groups,  it follows that Lemma \ref{lem:mesem2} follows by Lemma \ref{lem:mesem3} (in Lemma \ref{lem:mesem2} we are summing over $\gamma\in \Gamma$ and in Lemma \ref{lem:mesem3} the element $\gamma \in \Gamma$ is fixed). The crucial result here is Lemma \ref{lm:MeasureEstimateCrucial}. It allows to show that there are (sufficiently many) small translations of the set $S_{i,\gamma}$ which are disjoint. Therefore the measure of the set cannot be too large (this is made precise in the proof of Lemma \ref{lem:mesem3}). The crucial condition $\dim G-\dim(C(X))-3>0$ is used to show that the "sufficiently many" translates is enough to get the estimates since the cardinality of translates which are pairwise disjoint is strongly related to the number $\dim G-\dim(C(X))-3$ (see the statement of Lemma \ref{lm:MeasureEstimateCrucial})\footnote{It follows that Lemma \ref{lm:MeasureEstimateCrucial} is one of the main new tools which allows to generalize Ratner's results from \cite{Rat2}.}. In fact this is the only place in the whole proof of Theorem \ref{th1} in which we need this assumption.
The method of proving Lemma \ref{lm:MeasureEstimateCrucial} goes by a straightforward calculation using the properties of the adjoint representation of the $sl(2,\R)$ triple.

\textbf{Proof of Lemma \ref{lem:mesem}}.
Fix $\delta>0$ and $j\in \N$.  If $G(\delta, j,y)\neq \emptyset$, there exists $x\in G / \Gamma$ such that $x\in \Kak^{2,\delta}(2^j,\epsilon',y)$ and $\phi_p x\in \Kak^{2,\delta}(2^j,\epsilon',\phi_qy)$. By \eqref{emj2}
this implies that
\be\label{xbel}
x\in \exp(aV)\exp(bX)gy, \text{ where }
g\in \Bow(2^j,\epsilon',e).
\ee
with $|a|\in [2^{-j(1+10\delta)},2^{-j}]$, $|b|\leq \epsilon'$ and $a_V(g) = a_X(g) = 0$. Analogously,
$$
\exp(pU)x\in\exp(cV)\exp(dX)\tau(g')\exp(qU)y, \text{ where }
g' \in \Bow(2^j,\epsilon',e)
$$
with $c, d$ satisfy the same estimates as $a,b$ and $a_V(g') = a_X(g') = 0$.
We may lift $x$ and $y$ to $\tilde{x}$ and $\tilde{y}$ so that the first equation holds for the lifts as well. We may without loss of generality assume that $\tilde{y}$ is in a fixed fundamental domain $F$. Combining the two above equations and denoting $y_q=\phi_qy$ and $\tilde{y}_q = \exp(qU)\tilde{y}$, yields for some $\gamma \in \Gamma$:
\begin{multline*}\left(\exp(pU)\exp(aV)\exp(bX)\Bow(2^j,\epsilon',e)\exp(-qU)\right)\cap\\
\left(\exp(cV)\exp(dX)\Bow(2^j,\epsilon',e)\tilde{y}_q\gamma(\tilde{y}_q)^{-1}\right)\neq \emptyset.
\end{multline*}
In what follows below, we will conjugate equations by $\exp(sX)$ and $\exp(rU)$ and use the fact that they preserve measure, hence the measure of the set of $y$ for which the above holds will be equal to that of the conjugated equation.

Conjugating the above equation by $\exp(-sX)$, with $s=\frac{1}{2}\log (2^j)$ and using $[X,U]=2U$ and $[X,V]=-2V$,  implies that for $\tilde{y}' = \exp(-sX)\tilde{y}_q$ and some $\gamma\in \Gamma$, we have

\begin{multline}\label{eq:mestest}
\left(\exp(p'U)\exp(a'V)\exp(bX)\exp(-sX)\Bow(2^{j/2},\epsilon',e)\exp(sX)\exp(-q'U)\right)\cap\\
\left(\exp(c'V)\exp(dX)\exp(-sX)\Bow(2^{j/2},\epsilon',e)\exp(sX)\tilde{y}'\gamma(\tilde{y}')^{-1}\right)\neq \emptyset,
\end{multline}

where
\be\label{p'q}
p',q'\in [2^{20j\delta},2^{40j\delta}]\; |a'|,|c'|\in [2^{-10j\delta}, 1]\text{ and }|b|,|d|\leq \epsilon'.
\ee

Moreover by Lemma \ref{lem:B_Ggr} it follows that
$$\exp(-sX)\Bow(2^{j/2},\epsilon',e)\exp(sX)\subset \exp(C(U,X)_{\epsilon'})B_G(e,2^{-j/2}).
$$

where $C(U,X)_{\ve'} = B_{\mf g}(0,\ve) \cap C(U,X)$. Denote \[ h = \exp(-dX)\exp(-c'V)\exp(p'U)\exp(a'V)\exp(bX) \mbox{ \; \; and\; \; } \tilde{B}= h B_G(e,2^{-j/2}) h^{-1}.\]
Notice that for $w\in\{d,-c',a',b\}$ and $W\in \{U,V,X\}$ (since all the numbers are small), we have for every small enought $r>0$, $\exp(wW)B_G(e,r)\exp(-wW)\subset B_G(e,Cr)$ (with a global constant $C$). Moreover, by the bound on $p'$ (see \eqref{p'q}) it follows that \\
$\exp(p'U)B_G(e,C2^{-j/2})\exp(-p'U) \subset B_G(e, 2^{-j/2+K_Uj \delta})$, for some global constant $K_U$.\footnote{Here the reasoning follows from \eqref{eq:conj-formula} with $s=0$ and $t=p'$ .} Therefore, (by enlarging $K_U$ if necessary), we have
\be\label{eq:BallExchange}
\tilde{B}\subset B_G(e,2^{-\frac{j}{2}+K_Uj\delta}).
\ee

Therefore \eqref{eq:mestest}, \eqref{eq:BallExchange} and  Lemma \ref{lem:com} imply that
\be\label{eq:mesgaOld}
\tilde{y}'\gamma(\tilde{y}')^{-1}\in B_G(e,2^{-j(\frac{1}{2}-K_U\delta)})m(p',q',a',b,c',d)\exp(C(U,X)_{\epsilon'^{1/3}})
\ee
where $m(p',q',a',b,c',d)=\exp(-dX)\exp(-c'V)\exp(p'U)\exp(a'V)\exp(bX)\exp(-q'U)$.

From now, instead of the previous equation, we consider the square of the previous equation. The reason is that in the Appendix we do the computations in $SL(2,\R)$ (and then transfer to $G$) and this allows to deal with $-\id \in SL(2,\R)$.
Notice that considering the adjoint action of components of $m(p',q',a',b,c',d)$ one by one (see e.g. \eqref{eq:conj-formula} and \eqref{eq:conj-formula2} ) on the ball $B_G(e,2^{-j(\frac{1}{2}-K_U\delta)})$ and since the elements in $C(U,X)_{\epsilon'^{1/3}}$ are small, we have if $m = m(p',q',a',b,c',d)$, then
\be
\begin{aligned}
\left(m\exp(C(U,X)_{\epsilon'^{1/3}})\right)^{-1}B_G(e,2^{-j(\frac{1}{2}-K_U\delta)})m\exp(C(U,X)_{\epsilon'^{1/3}})\subset B_G(e,2^{-j(\frac{1}{2}-2K_U\delta)}).
\end{aligned}
\ee
Then our new equation is:
\be\label{eq:mesga}
\tilde{y}'\gamma^2(\tilde{y}')^{-1}\in B_G(e,2^{-j(\frac{1}{2}-2K_U\delta)})m^2(p',q',a',b,c',d)\exp(C(U,X)_{2\epsilon'^{1/3}}).
\ee


In the Appendix we will show that for every $p',q',a',b,c',d$ as above we have the following: there exists $K'_U>0$ (depending only on $U$) such that

\textbf{P1.} $d_G(m^2(p',q',a',b,c',d),e)< K'_Uj\delta$,

and

\textbf{P2.} $m^2(p',q',a',b,c',d)=h\exp(sX)h^{-1}$, where $40j\delta-5\leq|s|\leq 160j\delta+6$, $h$ commutes with $C(U,X)$ and
\be\label{pcon}
{h}^{-1}B_G(e,2^{-j(1/2-2K_U\delta)+1})h\subset B_G(e,
2^{-j(\frac{1}{2}-K'_U\delta))}).
\ee

Let $(m_i)_{i=1}^{T(j)}\in G$, 
be a $2^{-j/2}$ dense set in $m^2(p',q',a',b',c',d)$, i.e. for every $m^2$, there exists $m_i$ such that $d(m^2,m_i)\leq 2^{-j/2}$.
Notice that by \textbf{P1} and Lemma \ref{lem:vol-entropy} it can be done with
\be\label{tjbound}
T(j)\leq 2^{K''_U j\delta+3j/2}.
\ee

for some $K''_U$ depending only on $U$, since $B_G(x,\ve) \supset B_{SL(2,\R)}(x,\ve)$. Then  \eqref{eq:mesga} implies that for some $i\in \{1,\ldots T(j)\}$, we have
\be\label{eq:mesga2}
\tilde{y}'\gamma^2(\tilde{y}')^{-1}\in B_G(e,2^{-j(\frac{1}{2}-2K_U\delta)+1})m_i\exp(C(U,X)_{2\epsilon'^{1/3}}).
\ee

By \textbf{P2} it follows that $m_i=g_i\exp(s_iX)g_i^{-1}$ and moreover that $g_i$ commutes with $C(U,X)$. Therefore and by \eqref{pcon}, if we denote $\tilde{z}=g_i\tilde{y}'$, there is some $\gamma'$ such that
then \eqref{eq:mesga2} implies that
\be\label{eq:mesga3}
\tilde{z}\gamma'(\tilde{z})^{-1}\in B_G(e,2^{-j(\frac{1}{2}-K'_U\delta)})\exp(s_iX)\exp(C(U,X)_{
2\epsilon'^{1/3}})
\ee

Since  $T(j)\leq 2^{K''_U j\delta}(2^{3j/2})$, by the above reasoning\footnote{We have assumed \eqref{xbel} and transformed the equation to \eqref{eq:mesga3}. Hence the measure of $y$ solving \eqref{xbel} (and hence also belonging to the set $EG(\delta,j)$ Lemma \ref{lem:mesem}) is no larger than the measure of $y'$- solutions to \eqref{eq:mesga3}.}, Lemma \ref{lem:mesem} follows by the following lemma:

\begin{lemma}\label{lem:mesem2} For $j>j_\delta$ and for every $i\in \{1,\ldots T(j)\}$
\be\label{eq:lgbe}
\mu_F(\{\tilde{z}\in F: d_G(\tilde{z},e)<40j\delta\text{ and} \;\exists \gamma'\in \Gamma,\text{ such that } \tilde{z} \text{ satisfies } \eqref{eq:mesga3}\})\leq \frac{1}{T(j)j^3}.
\ee
\end{lemma}
\begin{proof}[Proof of Lemma \ref{lem:mesem}]
Notice that by the reasoning above, and Lemma \ref{lem:mesem2}, we have
\be
\begin{aligned}
&\mu((G/\Gamma)\setminus EG(\delta,j))\leq \mu_F(\{\tilde{z}\in F:d_G(\tilde{z},e)\geq 40j\delta\})+\\
&T(j)\mu_F(\{\tilde{z}\in F: d_G(\tilde{z},e)<40j\delta\;\text{and}\;\exists \gamma'\in \Gamma,\text{ such that } \tilde{z} \text{ satisfies } \eqref{eq:mesga3}\})\\&\leq
\mu_F(\{\tilde{z}\in F:d_G(\tilde{z},e)\geq 40j\delta\})+j^{-3}.
\end{aligned}
\ee
Moreover, by Corollary \ref{cor:exp-cusp-decay}, we have for some $c>0$
$$
\mu_F(\{\tilde{z}\in F:d_G(\tilde{z},e)\geq 40j\delta\})\leq 2^{-40cj\delta},
$$
what finishes the proof.
\end{proof}

We will now show Lemma \ref{lem:mesem2}:

Notice that if $\tilde{z} \in F$ is a solution of \eqref{eq:mesga3}, then by triangle inequality and the bound on $s_i$ (see \textbf{P2}) it follows that
$$
d_G(\gamma',e)\leq 2d_G(\tilde{z},e)+2+|s_i|\leq 81j\delta.
$$
By \cite{Gor-Nev} Theorem $1.7$ it follows that for some constant $C_\Gamma>0$
\be\label{cardg}
\left| \gamma\in \Gamma: d_G(\gamma,e)\leq 81 j\delta\right|\leq
2^{C_\Gamma j\delta}.
\ee
Define
\be\label{sigam}
S=S_{i,\gamma}:=\{\tilde{z} \in F , d_G(\tilde{z},e)<40j\delta:\:\text{ such that } \tilde{z}\text{ satisfies } \eqref{eq:mesga3} \text{ with }\gamma'=\gamma\}.
\ee
By \eqref{cardg}, Lemma \ref{lem:mesem2} follows by showing:

\begin{lemma}\label{lem:mesem3} For every $j \ge j_\delta$ and $\gamma\in \Gamma$, with $d_G(\gamma,e)<81 j\delta$ and every $i\in\{1,\ldots, T(j)\}$, we have
\be\label{eq:lgbe2}
\mu_F(S_{i,\gamma})\leq \frac{1}{j^3T(j)2^{C_\Gamma j\delta}}
\ee
\end{lemma}

For the remainder of the proof, we will let $U,X,V$ be denoted by $X_0^0$, $X_1^0$ and $X_2^0$ so as to simplify notation (our index on the superscript began at 1 before). In particular, $\set{X_l^k : k= 0,\dots,n \mbox{ and } l = 0,\dots, m_k}$ is a basis of $\mf g$. Observe that $C(X) = \langle X_{m_k/2}^k : k=1,\dots,n \mbox{ and }m_k \in 2\Z\rangle$. To prove Lemma \ref{lem:mesem3}, we need the following lemma:

\begin{lemma}\label{lm:MeasureEstimateCrucial}
There exists a constant $D_U$ (depending on $U$ and larger than $\frac{D_1}{40}$, where $D_1$ is specified in Proposition \ref{prop:measureOfPertubationSet}) such that  for $0\leq t_{kl},s_{kl}\leq 2^{-D_Uj\delta}$ for $k=1,\dots,n$ and $l = 0,\dots m_k$, $l\not= m_k/2$, if there exists $(k_0,l_0)$ such that $|t_{k_0l_0}-s_{k_0l_0}|\geq 2^{-j(\frac{1}{2}-D_U\delta)}$, then
\begin{equation}\label{eq:disjointProperty}
\exp\left(\sum_k\sum_{l\not=m_k/2}t_{kl}X_l^k\right)S_{i,\gamma}\cap\exp\left(\sum_k \sum_{l\not=m_k/2}s_{kl}X_l^k\right)S_{i,\gamma}=\emptyset,
\end{equation}
\end{lemma}
Before we prove Lemma \ref{lm:MeasureEstimateCrucial}, let us show how it implies Lemma \ref{lem:mesem3} and hence also finishes the proof of Lemma \ref{lem:mesem}.

\begin{proof}[Proof of Lemma \ref{lem:mesem3}] We take a maximal $2^{-j(\frac{1}{2}-D_U\delta)}$-separated set $A$ inside the cube

$[0,2^{-D_Uj\delta}]^{\dim G-\dim C(X)}$. Then by definition of maximal separated sets, we have
\be\label{cardA}
|A|\geq2^{(j(\frac{1}{2}-2D_U\delta)-1)(\dim G-\dim C(X))}
\ee

Suppose $s,t\in A$, then they differ in one direction by at least $2^{-j(\frac{1}{2}-D_U\delta)}$. Notice that since $(t_{kl})\in [0,2^{-D_Uj\delta}]^{\dim G-\dim C(X)}$, we have
$$\exp\left(\sum t_{kl}X_l^k\right)S_{i,\gamma}\subset \exp(B_{\mf g}(0,2^{-D_Uj\delta}))F_{40j\delta},$$
where $F_{t}:=\{z\in F:\; d_G(z,e)<t\}$.

Then we have the following proposition to control the measure of the set $$\exp(B_{\mf g}(0,2^{-D_Uj\delta}))F_{40j\delta},$$ which will be proved in Section \ref{sec:siegel-sets}:
\begin{proposition}\label{prop:measureOfPertubationSet} There exists a fundamental domain $F\subset G$ (for the lattice $\Gamma$) and two  constants $D_1,D_2>0$ such that for every for every $t$, we have
$$
 \mu_G\left( \exp(B_{\mf g}(0,2^{-D_1t}))F_t\right)<D_2.
$$
\end{proposition}

Therefore, by \eqref{eq:disjointProperty} and Proposition \ref{prop:measureOfPertubationSet}, we have
\begin{equation}\label{eq:estimateofC1}
\begin{aligned}
\mu_G(\exp(B_{\mf g}(0,2^{-D_Uj\delta}))F_{40j\delta})&\geq\mu_G\left(\bigcup_{(t_{kl})\in A}\exp\left(\sum t_{kl}X_l^k\right)S_{i,\gamma}\right)\\ &=\sum_{(t_{kl})\in A}\mu_G\left(\exp\left(\sum t_{kl}V_l^k\right)S_{i,\gamma}\right)\\&=|A|\mu_F(S_{i,\gamma}),
\end{aligned}
\end{equation}
Letting $C_\Gamma' = \mu_G(\exp(B_{\mf g}(0,2^{-D_Uj\delta}))F_{40j\delta})<D_2$, the above inequality and \eqref{cardA} implies:
\be\label{eqfer}
\mu(S_{i,\gamma})\leq C_\Gamma'|A|^{-1} \leq C_\Gamma'2^{-(j(\frac{1}{2}-2D_U\delta)-1)(\dim G-\dim C(X))}.
\ee
Since by assumption $\dim G-\dim(C(X))\geq 4$, for small enough $\delta$ (smallness depending only on constants $D_U,K'_U$ and $C_\Gamma$ ), we have
\begin{multline*}
(j(\frac{1}{2}-2D_U\delta)-1)(\dim G-\dim C(X))\geq 4(j (\frac{1}{2}-2D_U\delta)-1)\geq\\
[ C_\Gamma j\delta+K_U'j\delta+\frac{3}{2}j]+ \frac{j}{10}.
\end{multline*}
So by the bound on $T(j)$ (see \eqref{tjbound}), if $2^{j/10}> C_\Gamma'^{-1}j^3$ (which is always true for large $j$), we have
$$
C_\Gamma 2^{-(j(\frac{1}{2}-2D_U\delta)-1)(\dim G-\dim C(X))}\leq
\frac{1}{j^3T(j)2^{C_\Gamma j\delta}}.
$$
 This and \eqref{eqfer} finish the proof of Lemma \ref{lem:mesem3}.
\end{proof}

So it only remains to prove Lemma \ref{lm:MeasureEstimateCrucial}.
\begin{proof}[Proof of Lemma \ref{lm:MeasureEstimateCrucial}]
Suppose that for $t=(t_{kl}),s=(s_{kl})$ there is a $(k_0,l_0)$ such that $|t_{k_0l_0}-s_{k_0l_0}|\geq2^{-j(\frac{1}{2}-D_U\delta)}$ and suppose by contradiction that there is a $\tilde{x}$ (different than the $\tilde{x}$ from the start of this section) such that $$\tilde{x}\in \exp\left(\sum t_{kl}X_l^k\right)S_{i,\gamma}\cap\exp\left(\sum s_{kl}X_l^k\right)S_{i,\gamma}.$$

By definition (see \eqref{sigam} and \eqref{eq:mesga3}), that means there exists $g_1,g_2\in B(e,2^{-j(\frac{1}{2}-K'_U\delta)})$, $160j\delta+6>|\alpha|>40j\delta-5$ (in fact $\alpha=s_i$) and $h_1=\exp(Y_1)\text{, }h_2=\exp(Y_2)$ with $Y_1,Y_2\in C(U,X)_{2\epsilon'^{1/3}}$, such that
\begin{equation}
\begin{aligned}
&g_1h_1\exp(\alpha X)=\exp\left(-\sum t_{kl}X_l^k\right)\tilde{x}\gamma \tilde{x}^{-1}\exp\left(\sum t_{kl}X_l^k\right),\\
&g_2h_2\exp(\alpha X)=\exp\left(-\sum s_{kl}V_l^k\right)\tilde{x}\gamma \tilde{x}^{-1}\exp\left(\sum s_{kl}X_l^k\right).
\end{aligned}
\end{equation}
Computing $\tilde{x}\gamma \tilde{x}^{-1}$ from the second equation and plugging into the first one, we get
\begin{multline}\label{eq:lmInitialSetting}
g_1h_1\exp(\alpha X)= \\ \exp\left(-\sum t_{kl}X_l^k\right)\exp\left(\sum s_{kl}X_l^k\right)g_2h_2\exp(\alpha X)
\exp\left(-\sum s_{kl}X_l^k\right)\exp\left(\sum t_{kl}X_l^k\right).
\end{multline}

\begin{lemma}\label{eq:adeq}
Let
\[ M : (Y_1,\dots,Y_n) \mapsto \alglog(\exp(Y_1)\dots\exp(Y_2)) .\]

Then there exists $\kappa > 0$, $\sigma > 0$ such that if $\norm{(Y_1,\dots,Y_n)} < \sigma$, then $$\norm{M(Y_1,\dots,Y_n) - \sum Y_i} < \kappa \max_{i\not=j} \norm{Y_i}\cdot \norm{Y_j}.$$
\end{lemma}
\begin{proof}
This is an easy corollary of Taylor's theorem with the knowledge that $[Y_i,Y_i] = 0$ for every $i$.
\end{proof}

As $0\leq s_{kl},t_{kl}\leq 2^{-D_Uj\delta}$, by Lemma \ref{eq:adeq} it follows that there exists $K_U'''>0$ and $g_2'\in B(e,2^{-j(\frac{1}{2}-K''_U\delta)})$ such that
$$\exp\left(-\sum t_{kl}X_l^k\right)\exp\left(\sum s_{kl}X_l^k\right)g_2=
g_2'\exp\left(-\sum t_{kl}X_l^k\right)\exp\left(\sum s_{kl}X_l^k\right).$$

Using this and multiplying \eqref{eq:lmInitialSetting} by  $\exp(\alpha X)^{-1}$ from the right, we have
\begin{multline}\label{eq:lmMainTarget}
g_2'^{-1} g_1\exp(Y_1)=\exp\left(-\sum t_{kl}X_l^k\right)\exp\left(\sum s_{kl}X_l^k\right)\exp(Y_2)\exp(\alpha X)\cdot \\
\exp\left(-\sum s_{kl}X_l^k\right)\exp\left(\sum t_{kl}X_l^k\right)\exp(-\alpha X).
\end{multline}


We apply Lemma \ref{eq:adeq} to write $\exp\left(-\sum t_{kl}X_l^k\right) \exp\left(\sum s_{kl} X_l^k\right)$ as $$\exp\left( Y' + \sum(s_{kl} - t_{kl})X_l^k \right),$$ where $\norm{Y'} < \kappa \max\set{\abs{s_{k_1l_1}t_{k_2l_2}}}$. Notice that  since the chain basis elements are eigenvectors for $\ad_X$, we have by applying the conjugation to $\exp\left(Y' + \sum (s_{kl} - t_{kl})X_l^k\right)$:
\begin{multline}
\exp(\alpha X)\exp\left(-\sum s_{kl}X_l^k\right)\exp\left(\sum t_{kl}X_l^k\right)\exp(-\alpha X)=\\
\exp\left(Y'' + \sum (s_{kl}-t_{kl})e^{(m_k-2l)\alpha} X_l^k\right)
\end{multline}

where if $Y' = \sum y_{kl}X_l^k$, $Y'' = \sum e^{(m_k-2l)\alpha}y_{kl}X_l^k$, so

\[\norm{Y''-Y'} < \kappa\max\set{\abs{1-e^{(m_k-2l)\alpha}}}\max\set{\abs{s_{k_1l_1}t_{k_2l_2}}}.\]

 Therefore, since all terms are small

\begin{multline}\label{eq:combine}
\exp\left(-\sum t_{kl}X_l^k\right)\exp\left(\sum s_{kl}X_l^k\right)\exp(Y_2)\exp(\alpha X)\cdot \\
\exp\left(-\sum s_{kl}X_l^k\right)\exp\left(\sum t_{kl}X_l^k\right)\exp(-\alpha X)=\\
\exp\left( Y' + \sum (s_{kl} - t_{kl})X_l^k\right) \exp(Y_2) \exp\left(-Y'' - \sum (s_{kl}-t_{kl})e^{(m_k-2l)\alpha}X_l^k\right)  =\\
\exp\left( (Y'-Y'') +Y''' + \sum (s_{kl}-t_{kl})(1-e^{(m_k-2l)\alpha})X_l^k + Y_2\right)
\end{multline}
where

\begin{multline*}
\norm{Y'''} \le \kappa\max\left\lbrace \norm{Y' + \sum (s_i - t_j)V_i}\norm{Y_2},  \norm{Y'' - \sum (s_i-t_j)e^{\lambda_il}V_i } \norm{Y_2},\right. \\
\qquad \left. \norm{Y' + \sum (s_i - t_j)V_i}\norm{Y'' - \sum (s_i-t_j)e^{\lambda_il}V_i}\right\rbrace.
\end{multline*}

Thus if we set $Y = Y''' + Y' - Y''$, we get the following very rough bound since all terms in $Y'''$ and $Y' - Y''$ are products of terms bounded by this small number:
\begin{equation}\label{eqfff}
\|Y\|\leq \epsilon'^{1/100}\left\|
\sum (s_{kl}-t_{kl})(1-e^{(m_k-2l)\alpha})X_l^k\right\|.
\end{equation}
Finally notice that since $40j\delta-5<|\alpha|<160j\delta+6$ and by the choice of $k_0,l_0$ (recall $|s_{k_0l_0}-t_{k_0l_0}|\geq 2^{-j(\frac{1}{2}-D_U\delta)}$), we have
\begin{multline}\label{,000}
\max_{k,l} |(s_{kl}-t_{kl})(1-e^{(m_k-2l)\alpha})|\geq  |(s_{k_0l_0}-t_{k_0l_0})(1-e^{(m_{k_0}-2l_0)\alpha})|\geq\\
\frac{1}{2}|s_{k_0l_0}-t_{k_0l_0}| \geq \frac{1}{2}\cdot2^{-j(\frac{1}{2}-D_U\delta)}.
\end{multline}

Recall that $g_1\in B(e,2^{-j(\frac{1}{2}-K''_U\delta)})$, $g_2'\in B(e,2^{-j(\frac{1}{2}-K'''_U\delta)})$ and then there exists a constant $K''''_U>0$ only depends on chain structure such that $$g_2'^{-1} g_1\in B(e,2^{-j(\frac{1}{2}-K''''_U\delta)}).$$

Thus, by Lemma \ref{eq:adeq}, we have $${g_2}'^{-1}g_1\exp(Y_1)=\exp\left(Y_1+
\sum c_{kl}X_l^k\right)$$
where $|c_i|\leq
2^{-j(\frac{1}{2}-K'''_U\delta)}$.
It follows that the $X_k^l$ coefficients of the LHS of \eqref{eq:lmMainTarget} are less than $2^{-j(\frac{1}{2}-K'''_U\delta)}$ (since $Y_1\in C(U,X)\subset C(X)$). However, the coefficient by $k_0,l_0$ of the RHS of \eqref{eq:lmMainTarget} is by the definition of $Y$, \eqref{,000}, \eqref{eqfff} and \eqref{eq:combine}  bounded below by $\frac{1}{2}\cdot2^{-j(\frac{1}{2}-D_U\delta)}$, hence if $D_U$ is large enough (in terms of $K''''_U$ and larger than $\frac{D_1}{40}$) we get a contradiction with \eqref{eq:lmMainTarget}. This finishes the proof.
\end{proof}

\section{Coarse Fundamental Domains and Siegel Sets}
\label{sec:siegel-sets}

Let $G$ be a real semisimple Lie group and $\Gamma \subset G$ a lattice. We treat the case of rank one groups and higher-rank groups separately. Every semisimple group splits as a product of simple groups, and every lattice will split as a direct product of irreducible lattices after passing to a finite index subgroup. See Remark \ref{rem:irreducible-lattice}. Therefore, in this section, we assume that the lattice is irreducible in $G$ and not cocompact (notice that if $\Gamma$ is cocompact, Proposition \ref{prop:measureOfPertubationSet} follows trivially, as $F$ is a compact set), and we treat the case of $\rank(G) = 1$ and $\rank(G) > 1$. In both cases, we will seek something slightly weaker than a fundamental domain, which in our case is sufficient.

\begin{definition}
If $\Gamma \curvearrowright X$ is a properly discotinuous action of a discrete group on a metric space $X$, a {\it coarse fundamental domain} for the action is a subset $F \subset X$ such that  if $\pi : X \to X /\Gamma$ is the projection to the quotient,
\begin{enumerate}[(1)]
\item $\pi|_F$ is onto, and
\item $\set{ \gamma \in \Gamma : F \cdot \gamma \cap F \not= \emptyset}$ is finite.
\end{enumerate}

If $\Gamma \subset G$ is a discrete subset of a Lie group $G$, a coarse fundamental domain for $\Gamma$ is a coarse fundamental domain for the right action of $\Gamma$ on $G$.
\end{definition}

Notice that if the set in (2) is $\set{e}$. then $F$ is a fundamental domain. While (2) implies that $\pi$ is finite-to-one, it is slightly stronger (since the preimage can be reached by finitely many $\gamma \in \Gamma$ which are independent of $x \in F$).

\begin{remark}
\label{rem:irreducible-lattice}
Notice that because in the definition of a coarse fundamental domain, we only require that $\set{ \gamma \in \Gamma : F \cdot \gamma \cap F \not= \emptyset}$ is finite, if $\Gamma'$ is a finite index subgroup of $\Gamma$, and $\gamma_1,\dots,\gamma_s$ are representatives of $\Gamma / \Gamma'$, $F' = \bigcup_{i=1}^s F \gamma_i$ is a coarse fundamental domain for $\gamma_i$. Furthermore, because our estimates for measures are only designed to guarantee finiteness, we may assume that $\Gamma=  \Gamma_1 \times \dots \times \Gamma_n$ is a product of irreducible lattices in factor groups $G_i$, and producing coarse domains for each of the terms in the product $G_i / \Gamma_i$ will give a coarse comain for $\Gamma$. This justifies our assumption that $\Gamma$ is irreducible.
\end{remark}

\subsection{Geometry of Siegel Sets}

If $\rank_\R(G) = 1$, we define a Siegel set in the following way: fix a split Cartan subgroup $\R \cong A \subset G$. Then $\Lie(A)$ is generated by some unit vector $X$. Set $a_t = \exp(tX)$ and $A_+ = \set{ a_s : 0 < s < \infty}$. Let $K \subset G$ denote the maximal compact subgroup. There are two subgroups, $U_+$ and $U_-$, the stable and unstable subgroups, characterized by the property that $\ad_X$ preserves $\Lie(U_\pm)$ with only positive or negative eigenvalues, respectively. Given $\eta \subset A \cdot U_- =: P$ which is relatively compact in $P$, let $\mc S_{\eta} = K\cdot A_+ \cdot \eta$.

In the case of $\rank_\R(G) > 1$, because we have asume the lattice is irreducible, it must be arithmetic by the Margulis arithmeticity theorem. Therefore, after taking a compact extension if necessary of $G$, we may assume that $G = \mbf{G}(\R)$ and $\Gamma = \mbf{G}(\Z)$ for some $\Q$-algebraic group $\mbf{G}$ (since our original $\Gamma$ must be commensurable with $\mbf{G}(\Z)$ after taking a compact extension). Under these assumptions, one defines a Siegel set in the following way: Let $S \subset G$ be a maximal $\Q$-split torus. That is, $S$ is a maximal abelian subgroup which is diagonalizable over $\Q$ (ie, such that the corresponding $\Q$-subgroup is diagonalizable). Let $A \supset S$ be a maximal $\R$-split torus containing $S$.
$\mf a = \Lie(A)$ has a canonical set of weights $\Delta \subset \mf a^*$, and a splitting $\mf g = C_{\mf g}(\mf a)\bigoplus_{\alpha \in \Delta} \mf g^\alpha$, where $\mf g^\alpha = \set{ Y \in \mf g : \ad_X(Y) = \alpha(X) Y \mbox{ for all } X \in \mf a}$. Then there exists $a \in A$ such that if $\Delta_+ = \set{ \alpha \in \Delta : \alpha(a) > 0}$, $N$ is the simply connected subgroup with algebra $\bigoplus_{\alpha \in \Delta_+} \mf g^\alpha$. Let $\Delta_{S,+} = \set{\alpha \in \Delta_+ : \alpha|_S \not\equiv 0}$, and note that since $S$ is $\Q$-split, $\Delta_{S,+}$ consists of rational functionals. Let $P$ be the minimal $\Q$-parabolic subgroup containing $S$. Such a $P$ is the weak-stable manifold of some $a \in S$ acting on $G / \Gamma$. More explicitly:

\begin{equation}
\label{eq:parabolic-subgroup}
 P = \set{ g \in G : d_G(a^n,a^ng) < \infty \mbox{ for all }n \in \Z_+ }
\end{equation}

We then build the associated {\it Weyl chamber} $S_- \subset S$ corresponding to $\Delta_{S,+}$. $S_-$ is exactly the set $\set{a \in S : \beta(a) \le 0 \mbox{ for all }\beta \in \Delta_{S,+}}$. Let $S_t = \set{ a \in S : \beta(a) \le t \mbox{ for all }\beta \in \Delta_{S,+}}$. 

\begin{definition}
Given $S_-$ and $B$ a positive Weyl chamber and corresponding minimal parabolic subgroup as defined above, and a relatively compact subset $\eta \subset P$, the {\it Siegel set} for $\eta$ is the set $\mc S_{t,\eta} = K \cdot S_t \cdot \eta$.
\end{definition}

Siegel sets are the basic building blocks of coarse fundamental domains.  This follows for certain classical groups and lattices from the early works of Siegel \cite{siegel}, for rank one groups from Garland and Raghunathan, and for higher-rank groups (where lattices are known to be arithmetic) by works of Borel \cite{borel} and Harish-Chandra \cite{borel-harish-chandra}. An accessible summary of this topic can be found in \cite[Chapter 19]{witte-morris}.

\begin{theorem}[Garland, Raghunathan, Siegel, Borel, Harish-Chandra]
\label{thm:coarse-domain}
There exists some $t \in \R$ and $\eta \subset P$, and $b_1,\dots,b_n \in G$ such that $\bigcup_{i=1}^n \mc S_{t,\eta} \cdot b_i$ is a coarse fundamental domain for $\Gamma \subset G$.
\end{theorem}

The following is classical:

\begin{lemma}
\label{lem:siegel-finite}
For any $t \in \R$ and $\eta \subset P$, $\mu(\mc S_{t,\eta}) < \infty$.
\end{lemma}

We recall a sketch of the proof, aspects of which we shall use later. For a complete proof, see, for instance, \cite[Lemma 12.5]{borel}

\begin{proof}[Sketch of Proof]
Notice that the map $\pi : K \times S \times P \to G$ defined by $\pi(k,s,p) = ksp$ is onto $G$. Furthermore, if $dk$, $ds$ and $dp$ are corresponding Haar measures on $K$, $S$ and $P$, respectively, then $dg = ce^{\rho(s)} \pi_*(dk \wedge ds \wedge dp)$, where $\rho(s) = \sum_{\alpha \in \Delta_{S,+}} \alpha(s)$. One sees this since the pushforward measure will be invariant under right translations by $s$, $p$ and left translations by $k$. Let $\tilde{S} = \set{ s \in S_- : \rho(s) = -1}$, so that any $s \in S_-$ is a multiple of some $s_0 \in \tilde{S}$. Therefore, the measure of $\mc S_{t,\eta}$ is at most a constant (the total measure of $K$) times $\int_{-\infty}^t e^{\rho(t)} \vol(\tilde{S})\cdot t^{\dim(S)-1} \cdot \vol(\eta) \, dt$. This is clearly finite.
\end{proof}

Given $x \in \mc S_{t,\eta}$, we may write $x = ksu$, with $k \in K$, $s = \exp(X) \in S$ and $u \in \eta$. Define $\alpha(x) = \min \norm{\ad_X|_{\Lie(P)}}$, where the minimum is taken over any such presentation of $x$. Any two such presentations for $s$ must differ by varying the choice of $k$ and $p$ over compact sets. Therefore, given any such a presenation $\alpha(x) \ge \norm{\ad_X|_{\Lie(P)}} - \sigma$ for some fixed $\sigma$ which depends only on $\eta$. Notice that since $\Delta_+$ contains a basis of $A^*$, $\Delta_{S,+}$ contains a basis of $S_+$. Therefore, $\alpha|_S$ acts like an $L^\infty$ norm, but fails to be a norm only by a constant $\sigma$ (by identical reasons to non-uniqueness of presentations as descibed above). That is, there exists $\lambda,\sigma > 0$ (with $\lambda$ depending only on the choice of norm, and $\sigma$ depending on the choice of norm and the choice of $\eta$) such that

\begin{equation}
\label{eq:alpha-norm}
\lambda^{-1}\alpha(\exp(X))-\sigma \le d_G(e,\exp(X)) = \norm{X} \le \lambda \alpha(\exp(X)) +\sigma.
\end{equation}

\begin{lemma}
\label{lem:siegel-balls-rank1}
For any $\eta \subset P$, there exists $\kappa > 0$ and $\eta \subset \eta' \subset P$ such that:

\[ \bigcup_{x \in \mc S_{\eta}} B(x, \kappa e^{-\alpha(x)}) \subset \mc S_{\eta'}.\]
\end{lemma}

\begin{proof}
Let $\eta' = B(\eta,r)$ be the ball of radius $r$ around $\eta$. If $x \in \mc S_{t,\eta}$, we may write $x$ as $x = k  \cdot s \cdot u$, with $k \in K$, $s \in S_t$, $u \in \eta$. Then if $y \in B(x,e^{-\alpha(x)})$, since $K \times S \times P \to G$ is an open map, $y = k' \cdot s' \cdot u'$, with $k' \in K$, $s' \in S$ and $u' \in P$, and each $k'$, $s'$ and $u'$ are close to $k$, $s$ and $u$, respectively (we will examine the degree of closeness soon). Notice that:

\[ d_G(x,y) = \norm{\log(yx^{-1})} = \norm{\log(k'(k)^{-1}\cdot k(s's^{-1})k^{-1}\cdot k(s(u'u^{-1})s^{-1})k^{-1})} \]

Since conjugation by $k$ is an isometry and $S$ normalizes $P$ the above expression is in the image of $K \times kSk^{-1} \times kPk^{-1}$ and within $\ve$ of $e$. Therefore,, we get that if $d_G(x,y) < e^{-\alpha(x)}$, then $d_G(k,k') < ce^{-\alpha(x)}$,  $d_G(s,s') < c e^{-\alpha(x)}$ and $d_G(su,su') < c e^{-\alpha(x)}$ for some $c$. Then since $\alpha(x) \geq \norm{\ad_{\log(s)}|_{\Lie(B)}} - \sigma'$, we get that $d_G(u,u') < c$. Therefore, if we take $r = c \cdot( \abs{t} +  \max \set{\alpha(p) : p \in \eta}) < \infty$, we get the result.



\end{proof}


\begin{lemma}
\label{lem:quasi-isom1}
There exists $c,L > 0$ such that if $x \in \mc S_{t,\eta}$, $\alpha(x) \le Ld_G(x,e) + c$.
\end{lemma}

\begin{proof}
Let $c' = \diam(K) + \diam(\eta)$. Notice that $t \mapsto \exp(tX)$ is geodesic in $G / \Gamma$ for any unit vector $X \in S_-$, so $d_G(\exp(tX),e) = t$. Then if $x = kau$ with $k \in K$, $a = \exp(tX) \in S_-$ and $u \in \eta$:

\begin{eqnarray*}
\alpha(x) & \le & \lambda t  + \sigma \lambda\\
 & = & \lambda d_G(e,a) + \sigma \lambda\\
 & \le & \lambda (d_G(k,ka) - d_G(e,k) - d_G(ka,kau) + c') + \sigma\lambda \\
 & \le & \lambda (d_G(e,x) + c') + \sigma \lambda
\end{eqnarray*}

The first inequality follows from \eqref{eq:alpha-norm}, the second from the choice of $c'$ and the third from the reverse triangle inequality.
\end{proof}

\begin{proof}[Proof of Proposition \ref{prop:measureOfPertubationSet}]
Fix a coarse fundamental domain which is a union of Siegel sets, and choose a fundamental domain contained in the coarse fundamental domain. Then each Siegel set can be expanded by Lemma \ref{lem:siegel-balls-rank1} to include balls which decay at exponential rates according to the function $\alpha(x)$. Then if $x$ belongs to the coarse fundamental domain, $x = b_i x'$ for some $b_i$ in the finite set of Theorem \ref{thm:coarse-domain} and $x' \in S_{\eta}$. Let $D'$ be the Lipschitz constant of multiplication by $b_i$. Then if $y \in B(x,De^{-Ld_G(e,x)}) \subset b_iB(x,D'De^{-Ld_G(e,x')})$, by Lemma \ref{lem:quasi-isom1}, $y \in B(x,De^{-(\alpha(x) + c)}) \subset b_iB(x',e^{-\alpha(x')}) \subset b_iS_{\eta'}$ for sufficiently small $D$ (one easily sees that $\abs{\alpha(x) - \alpha(x')}$ is uniformly bounded above and below). Since $\mu(S_{\eta'}) <\infty$, we get the desired result.
\end{proof}

\begin{corollary}
\label{cor:exp-cusp-decay}
If $F$ is any fundamental domain chosen inside a coarse fundamental domain obtained from Theorem \ref{thm:coarse-domain}, then there exists $c > 0$ and $\kappa > 0$ such that  \[\mu(\set{ z \in F : d_G(e,z) \ge t}) \le c e^{-\kappa t}.\]
\end{corollary}

\begin{proof}
It suffices to show the claim for a single Siegel set. By Lemma \ref{lem:quasi-isom1}, it suffices to replace the set with $S_{t,\eta}$. But the proof sketch of Lemma \ref{lem:siegel-finite}, we saw that this was given by the integral of an exponentially decaying function times a polynomial, which was exponentially decaying. Therefore, we conclude the desired decay rate.

\end{proof}

\appendix
\section{Proof of \textbf{P1} and \textbf{P2}.}
Recall that we have the homomorphism $\phi:\mf{sl}(2,\R)\to \mf{g}$ taking the standard horocyclic generator of $\mf{sl}(2,\R)$ to $U$. 
By Lemma \ref{lem:sl2-lift}, $\phi$ extends to $\tilde{\phi}$. Therefore to prove \textbf{P1} and \textbf{P2} it is enough to make the computations in $SL(2,\R)$.

Recall that
\be\label{masd}m(p,q,a,b,c,d)=\exp(-dX)\exp(-cV)\exp(pU)\exp(aV)\exp(bX)\exp(-qU),
\ee
where $U,V,X$ with
\be\label{eq:newqwe}
p,q\in [2^{20j\delta},2^{40j\delta}],\abs{a},\abs{c}\in [2^{-10j\delta},1], |b|,|d|\leq \epsilon '.
\ee
We can consider this matrix in $SL(2,\R)$ or $G$ equivalently. We have
\begin{lemma}\label{lem:eds} For every $p,q,a,b,c,d$ as in \eqref{eq:newqwe}, we have that
$m^2=h\exp(sX)h^{-1}$, where $|s|\in [40j\delta-5,160j\delta+6]$ and $h=k\exp(\alpha U)$, with $k \in SO(2)$ and $\abs{\alpha} \le 2^{K_U'\delta j}$ for some fixed $K_U'$
\end{lemma}
Before we prove the above Lemma let us show how it implies \textbf{P1} and \textbf{P2}:

\begin{proof}[Proof of \textbf{P1} and \textbf{P2}]
Notice that by right invariance of $d$ and triangle inequality, we have
\begin{multline*}
\frac{1}{2}d(m^2,e)\leq d(m,e)\leq  d(\exp(-dX),e)+d(\exp(-cV),e)+ \\d(\exp(pU,e)+d(\exp(aV),e)+d(\exp(bX),e)+d(\exp(-qU),e)
\end{multline*}
Notice that by \eqref{eq:newqwe} it follows that the terms of the RHS with $d,c,a,b$ are bounded. Moreover, by Lemma \ref{lem:dist-computation} it follows
that
$$d(\exp(pU),e),d(\exp(-qU),e)\leq C\max(\log p,\log q)\leq 40Cj\delta.
$$
 This finishes the proof of \textbf{P1}.

Notice also that the first part of \textbf{P2} follows from Lemma \ref{lem:eds} with\\
 $h = k \exp(\alpha U)$. Let us now show \eqref{pcon}. Notice that
\begin{multline}
k\exp(\alpha U) B_G(e,2^{-j(1/2-2K_U\delta)+1})\exp(-\alpha U)k^{-1} = \\ \exp(\alpha U) B_G(e,2^{-j(1/2-2K_U\delta)+1})\exp(-\alpha U).
\end{multline}

Futhermore, by the bound on $\alpha$ and \eqref{eq:conj-formula}, we get that the right hand side above is contained in $B_G(e,2^{-j(1/2-K_U''\delta)})$ for sufficiently large $j$. This finishes the proof.
\end{proof}

So we only need to prove Lemma \ref{lem:eds}
\begin{proof}[Proof of Lemma \ref{lem:eds}]
Then by direct computation, we have
\begin{equation}
m=\left(
\begin{array}{cc}
 e^b \left(a e^{-d} p+e^{-d}\right) & e^{-b-d} p-e^b \left(a e^{-d} p+e^{-d}\right) q \\
 e^b \left(a \left(e^d-c e^d p\right)-c e^d\right) & e^{-b} \left(e^d-c e^d p\right)-e^b \left(a \left(e^d-c e^d p\right)-c e^d\right) q \\
\end{array}
\right).
\end{equation}

The trace of this matrix is $$e^{-b-d} \left(q e^{2 (b+d)} (a (c p-1)+c)+e^{2 b} (a p+1)+e^{2 d} (1-c p)\right).$$
Notice that $\abs{acpq}e^{b+d}$ is the dominating term above (see \eqref{eq:newqwe}), we have $\frac{9}{10}\abs{acpq}e^{b+d} \le \Tr(m) \le \frac{11}{10}\abs{acpq}e^{b+d}$. Therefore,
\begin{equation}\label{eq:MEstimateTrace}
2^{20j\delta-1}\leq|\operatorname{Tr}(m)|\leq 2^{80j\delta+1}.
\end{equation}

Since $|\operatorname{Tr}(m)|>2$, $m$ is diagonalizable. As a result we have $\operatorname{Tr}(m)=\lambda^{-1}+\lambda$ (suppose $|\lambda|>1$), with $2^{20j\delta-2}\leq|\lambda|\leq 2^{80j\delta+2}$. The estimate of $m^2$'s eigenvalue will follow from this.

So $m$ can be diagonalized as $m = h'\begin{pmatrix} \lambda & \\ & \lambda^{-1} \end{pmatrix}h'^{-1}$. Write $h' = kan'$, with $k \in SO(2,\R)$, $a$ a diagonal matrix, and $n' = \exp(\alpha' U)$ for some $\alpha' \in \R$. Then $h' = kna$, and notice that if $\pm h'\exp(sX)h'^{-1} = m$, then if $h = kn$, $\pm h\exp(sX) h^{-1} = m$.

Now, $$d(\pm h\exp(sX)h^{-1},e) = d(\pm kn\exp(sX)n^{-1}k^{-1},e) = d(\pm n\exp(sX)n^{-1},e).$$ If $n = \exp(\alpha U)$, then $n\exp(sX)n^{-1} = \exp((1-\lambda^{-2})\alpha U)\exp(sX)$, so

\begin{equation}
\begin{aligned}
d(\pm \exp((1-\lambda^{-2})\alpha U),e) &= d(n\exp(sX)n^{-1}\exp(-sX),e) \\ &\le d(n\exp(sX)n^{-1},e) + d(\exp(-sX),e) \\&=  d(m,e) + \abs{s} \\&\le 160Cj\delta + 160j\delta+4.
\end{aligned}
\end{equation}

On the other hand, $d(\pm \exp((1-\lambda^{-2})\alpha U),e) \ge c\log \abs{(1-\lambda^{-2})\alpha}-\pi \ge c\log \abs{\alpha} + c\log\abs{1-\lambda^{-2}}-\pi$ (we get a $\pi$ because we may need to multiply by $-\id$ which has distance $\pi$ to $e$). Therefore, $\log \abs{\alpha} \le 2^{K'_Uj\delta}$.


\end{proof}

\end{document}